\documentclass[english]{article}
\usepackage[T1]{fontenc}
\usepackage[letterpaper]{geometry}
\usepackage{float}
\usepackage{amsmath}
\usepackage{graphicx}
\usepackage{color}
\usepackage{epstopdf}

\usepackage[colorlinks, hypertexnames=false]{hyperref}
\hypersetup{linkcolor=blue, citecolor=red} 

\usepackage{adjustbox}
\usepackage{bm}
\usepackage{multirow}
\usepackage{array}
\usepackage{mathrsfs}
\usepackage{amsmath, amssymb, amsthm}
\usepackage{subfigure}
\usepackage{tikz}

\newcommand{\squad}{\hspace{0.1em} }

\graphicspath{{./Figures/}}

\newtheorem{theorem}{Theorem}

\newtheorem{lemma}{Lemma}
\newtheorem{remark}{Remark}

\usepackage{amsfonts} 
\usepackage{color}
\usepackage{adjustbox}
\definecolor{black}{rgb}{0,0,0}

\definecolor{red}{rgb}{1,0,0}

\definecolor{blue}{rgb}{0,0,1}

\usepackage{multirow}

\makeatother

\usepackage{babel}
\usepackage{authblk}
\title{}

\title{\textbf{}}
\title{\textbf{Generalized multiscale finite element method for highly heterogeneous compressible flow}}
\author{Shubin Fu\thanks{Department of Mathematics, University of Wisconsin-Madison, WI, USA. }, \quad
	Eric Chung\thanks{Department of Mathematics, The Chinese University of Hong Kong, Shatin, Hong Kong SAR.}, \quad and\quad
	Lina Zhao\thanks{Department of Mathematics, City University of Hong Kong, Kowloon Tong, Hong Kong SAR. Corresponding author (linazha@cityu.edu.hk).} \;
	
}

\begin{document}
	\maketitle
	\begin{abstract}
	In this paper, we study the generalized multiscale finite element method (GMsFEM) for single phase compressible flow in highly heterogeneous porous media. We follow the major steps of the GMsFEM to construct permeability
	dependent offline basis for fast coarse-grid simulation. The offline coarse space is efficiently constructed only once based on the initial permeability field with parallel computing. A rigorous convergence analysis is performed for two types of snapshot spaces. The analysis indicates that the convergence rates of the proposed multiscale method depend on the coarse meshsize and the eigenvalue decay of the local spectral problem.
	To further increase the accuracy of multiscale method, residual driven online multiscale basis is added to the offline space. The construction of online multiscale basis is based on a carefully design error indicator motivated by the analysis.
	We find that online basis is particularly important for the singular source.
	Rich numerical tests on typical 3D highly heterogeneous medias are
	presented to demonstrate the impressive computational  advantages of the
	proposed multiscale method.
	
\end{abstract}

\textbf{Keywords:} GMsFEM; compressible flow; highly heterogeneous; spectral problem; residual driven online multiscale basis.

\pagestyle{myheadings} \thispagestyle{plain} \markboth{Lina}
    {GMsFEM for highly heterogeneous compressible flow}

\section{Introduction}
Fluid modeling through heterogeneous porous media is important in  applications
such as reservoir simulation, nuclear water storage and
underground water contamination. These problems can be very challenging due to
the strong heterogeneities of the geological data. In flow simulation based inverse problems such as history matching,  one
needs to repeatedly solve the forward problems which makes the
full scale simulations almost impossible and motivates  intensive research on
model reduction techniques. There are two types of model reduction approaches, one is upscaling \cite{wu2002analysis,durlofsky1991numerical,arbogast2013ms}, in which
the upscaled geological properties such as permeability fields are
obtained based on some rules, therefore one can solve the problem with a much reduced model. Another direction is multiscale method \cite{hou1997multiscale,efendiev2013generalized,chung2015mixed,maalqvist2014localization,chen2003mixed,araya2013multiscale,jenny2003multi,weinan2007heterogeneous,ArPeWY07,aarnes04,jenny2003multi,hajibeygi2009multiscale,chen2020random},
in which one aims to solve the problem in  coarse grid with carefully constructed
multiscale basis functions. Notable multiscale methods include the multiscale method  (MsFEM) \cite{hou1997multiscale,chen2003mixed}, the generalized multiscale finite element method (GMsFEM) \cite{efendiev2013generalized,chung2015mixed},
the multiscale finite volume method (MsFVM) \cite{hajibeygi2009multiscale,jenny2003multi,efendiev2006accurate},
the heterogeneous multiscale method (HMM) \cite{weinan2007heterogeneous}, the
variational multiscale method (VMS) \cite{hughes98},  the multiscale mortar mixed finite
element method (MMMFEM) \cite{ArPeWY07}, the localized orthogonal method (LOD) \cite{maalqvist2014localization}, the multiscale hybrid-mixed method (MHM) \cite{araya2013multiscale} and recently proposed multiscale methods with randomized sampling \cite{chen2020random}. All these multiscale methods generate reasonably satisfactory numerical results in certain applications.

Among these multiscale methods, the MsFEM and its extension GMsFEM have achieved huge success in many practical applications especially fluid simulation arisen from   reservoir simulation.  Pioneer work in the MsFEM can be traced back to \cite{babuvska1994special}, in which the authors used special basis functions to replace the polynomials for second order elliptic problems with rough coefficients.
The idea was then extended to high-dimensional case in \cite{hou1997multiscale} and induced
a series of follow-up work. In \cite{chen2003mixed}, the mixed multiscale finite element method was developed and successfully applied for incompressible two-phase flow simulations,  which lead to vast research  on multiscale methods \cite{hajibeygi2009multiscale,chung2015mixed,jenny2003multi,wheeler2011multiscale,ArPeWY07,aarnes04,fu2019local,rocha2020multiscale,lie2019introduction} for reservoir
simulation. The key idea of the MsFEM is to construct multiscale basis functions via solving local problems with appropriate boundary conditions, these multiscale basis
functions contain important local media information and thus yield accurate
coarse-grid solution. However, the MsFEM fails to handle arbitrarily complicated media, which motivates the development of the GMsFEM \cite{efendiev2013generalized,chung2016adaptive}.
Carefully designed spectral problems are exploited to construct the multiscale basis in GMsFEM, therefore multiple basis functions are allowed in GMsFEM and thus the accuracy of the GMsFEM solution can be tuned and controlled. Another major highlight of the GMsFEM is its ability to cope with any types of heterogeneous media. To further improve the
performances of  GMsFEM, residual driven online multiscale basis functions were proposed
in \cite{online_cg}, these multiscale basis contains local and global media information and source information, which significantly facilitate  convergence of the
multiscale method. It is observed for time dependent problems or nonlinear problems, one can reuse the residual driven multiscale basis functions computed at certain time step or iteration
\cite{chung2017residual,fu2020constraint,fu2019generalized,wang2021local,wang2021comparison}, which  tremendously increase the accuracy of the GMsFEM solution compared with only using
equal numbers of offline basis functions.

We adopt the basic ideas of the offline and online GMsFEM for the
single phase nonlinear compressible flow arisen from reservoir simulation in this article. Most of the existing works in the context of GMsFEM are focused on
the incompressible flow, see e.g. \cite{chung2015mixed,wang2021comparison,wang2020online}.
Using other types of multiscale methods for the compressible flow can be found in \cite{KimPark07,hajibeygi2009multiscale,arraras2019multipoint,lie2012mixed,pal2015validation,Arshad21}, but there is few evidence that these methods (most of which are the variants of MsFEM) can deal with arbitrarily complicated porous media. Therefore it is necessary to systemically study the GMsFEM for the nonlinear compressible flow in  high-contrast media. We follow the major steps of the offline GMsFEM and online GMsFEM method.
In particular, we construct the permeability dependent offline multiscale
basis functions by solving local spectral problems with initial permeability field.
Although the single phase compressible flow is a time dependent problem and the
permeability field changes at different time instants, the multiscale space will keep fixed as time marches, which is a typical strategy in flow simulations \cite{chung2015mixed,aarnes04}. As a result, the CPU time for the offline stage can be neglected especially  the parallel computing can be employed without too much difficulty for solving the independent local problems. To boost the performance of the coarse-grid simulation especially when the source term is singular, which is often the case in practice, the residual driven online multiscale basis are incorporated here. We compute the online basis with the residual at the initial time step and selectively update them at later time steps to balance the accuracy and computational cost. The convergence of the semi-discrete formulation based on two types of snapshot spaces is rigorously analyzed. More specifically, we first bound the error between the fine scale solution and the coarse scale solution by the difference between the fine scale solution and its corresponding projection to the coarse grid. Then we analyze the error between the fine scale solution and its corresponding projection. To guide the construction of online basis functions, we also prove the a posterior error estimates. It is worth mentioning that a rigorous convergence error estimates for GMsFEM with applications to nonlinear problem is rarely seen in the existing literature and the proposed analysis and algorithm will definitively inspire more works in this direction.

Extensive numerical experiments  are provided to show the superior computational
performances of the proposed method in terms of CPU time. In particular, we consider
various  3D highly heterogeneous permeability fields with two types of boundary conditions and source settings. We are particularly interested in investigating the influence of adding offline and online multiscale basis on the accuracy of the multiscale solution. We report detailed CPU time for both the multiscale simulation and fine grid simulation to quantify reduction of the computational cost of the GMsFEM. It is shown that adding online basis is more effective than adding equal number of offline basis in reducing the error of GMsFEM solution especially if the source is singular. Besides, updating online basis can accelerate the convergence of the GMsFEM. The Newton's method is carried out  to handle the nonlinear term
and it is shown only a small number of iterations are needed in each time marching step.

The rest of the paper is organized as follows. In the next section, we introduce the
single phase compressible flow model with some preliminary results. The construction of the offline multiscale space and resulting GMsFEM algorithm are presented in Section \ref{sec:offline} and the corresponding convergence is
shown in Section \ref{sec:convergence}. We then introduction the residual driven basis and related analysis in Section \ref{sec:online}.
Numerical experiments are presented in Section \ref{sec:na}. We conclude the paper in Section \ref{sec:conclusion}.

	\section{Preliminaries.}\label{sec:pre}
	We consider the following single-phase nonlinear compressible flow \cite{ctene2015adaptive} through a
	porous medium:
	\begin{equation}\label{eq:model}
	\begin{aligned}
	\frac{\partial (\phi\rho)}{\partial t}-\nabla\cdot\big(\frac{\kappa}{\mu}\rho
\nabla p\big)=&q \quad \text{in}\squad D\times (0,T],\\
	\frac{\kappa}{\mu}\rho
	\nabla p\cdot n=&0 \quad \text{on}\squad \partial D^{n}\times(0,T],\\
	p=&p^d \quad \text{on}\squad \partial D^{d}\times(0,T],\\
	p=&p_0 \quad \text{on}\squad D\times\{t=0\}.
	\end{aligned}
	\end{equation}
	Here, $p$ is the fluid pressure that we aim to seek,
	$\mu$ is the constant fluid viscosity, $\phi$ is the porosity which is assumed to be a constant in our presentation.
	$\kappa$ is the permeability field that may be highly heterogeneous. $D$ is the computational domain, $\partial D=\partial D^n\cup \partial D^d$,
	$n$ is the outward unit-normal vector on $\partial D$.
	 The fluid density $\rho$ is a function of
fluid pressure $p$ as
\begin{equation}
	\rho(p)=\rho_{\text{ref}}	e^{c(p-p_{\text{ref}})},
\end{equation}
where $\rho_{\text{ref}}$ is the given reference density and  $p_{\text{ref}}$
is the reference pressure.

In the GMsFEM considered in this paper, multiscale basis functions will be
constructed for the pressure $p$. For later use, we first introduce the
notion of the two-scale mesh.
We divide the computational domain $D$ into some regular coarse blocks
and denote the resulting triangulation as $\mathcal{T}^H$. We use $H$ to represent the diameter of the coarse block $K\in \mathcal{T}_H$. Each coarse block will be
further divided into a connected union of fine-grid blocks which are conforming
across coarse-grid edges. We denote this fine-grid partition as $\mathcal{T}^h$, which is a refinement of $\mathcal{T}^H$ by definition.
For each vertex $ {x}_i \in \mathcal{S}^H$ in the grid  $\mathcal{T}^H$,  the coarse neighborhood $\omega_i$ is defined by
\begin{equation*}
\omega_i = \bigcup \{ K_j \; : \; K_j \subset \mathcal{T}^H, \;  {x}_i \in K_j \}.
\end{equation*}
That is, $\omega_i$ is the union of all coarse grid blocks $K_j$
containing the vertex $ {x}_i$, see Figure \ref{fig:grid}. The  multiscale basis functions are constructed  in each coarse neighborhood $\omega_i$. Throughout the paper, $a\preceq b$ means there exists a positive constant $C$ independent of the meshsize such that $a\leq C b$. In addition, $(\cdot,\cdot)$ stands for the standard $L^2$ inner product defined on the domain $D$.
\begin{figure}[h]
	\centering
	\includegraphics[width=3in,height=2in]{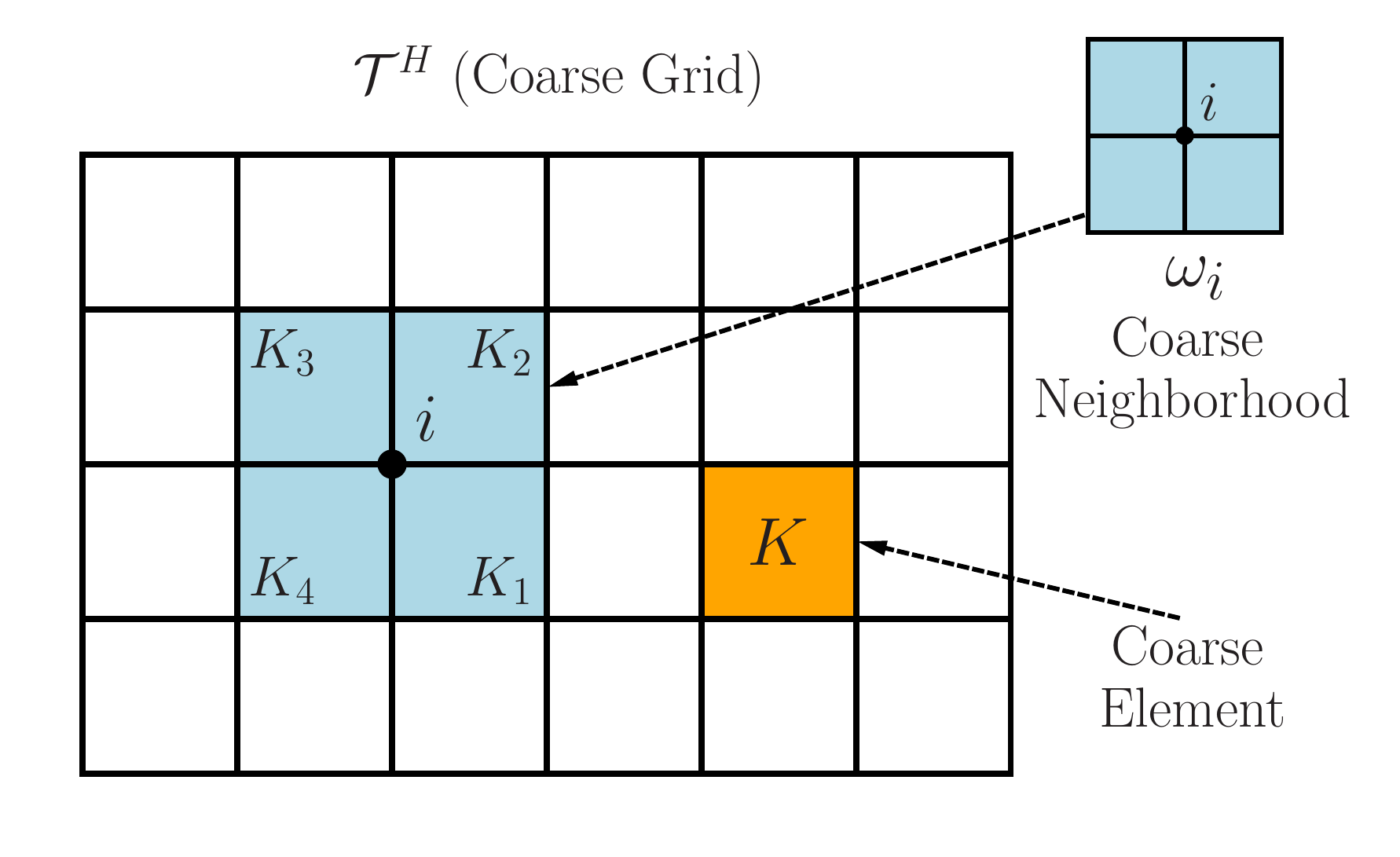}
	\caption{Illustration of coarse neighborhood and coarse element.}
	\label{fig:grid}
\end{figure}

Let $V^h$ be the space of the first-order Lagrange function with respect to the
fine-grid mesh $\mathcal{T}^h$. Then the finite element approximation to \eqref{eq:model} on the fine grid is to seek
\begin{align}
(\frac{\partial}{\partial t}(\phi \rho(p_h)),v)+(\frac{\kappa}{\mu}\rho (p_h)\nabla p_h,\nabla v)=(q,v)\quad \forall v\in V_h.\label{eq:fine}
\end{align}

To derive the fully discrete scheme for \eqref{eq:fine}, we introduce a partition of the time interval $[0,T]$ into subintervals $[t^n,t^{n-1}]$, $1\leq n\leq N_t$ ($N_t$ is an integer) and denote the time step size by
$\delta_t^n=t^n-t^{n-1} $. Then using backward Euler scheme in time, we can obtain the fully discrete scheme as follows: Find
$p_h^n$ such that
\begin{equation}\label{eq:femapp}
(\phi\rho(p_h^n),w_h)-(\phi\rho(p_h^{n-1}),w_h)+\delta t^n(\frac{\kappa}{\mu}\nabla \rho(p_h^n),\nabla w_h)-\delta t^n(q,w_h)=0\quad \forall w_h\in V_h.
\end{equation}
The nonlinear equation \eqref{eq:femapp} can be solved by the Newton's method.
Specifically, let $\{\phi_i\}_{i=1}^{i=N_f}$ be the finite element basis functions for $V_h$,  we now can write $p_h^{n,k}=\sum_{i}p_i^{n,k}\phi_i$ and $p_h^{n-1}=\sum_{i}p_i^{n-1}\phi_i$, $k$ denotes the $k$-th Newton iteration.
Then, we can recast the nonlinear equation \eqref{eq:femapp} as a residual equation
system:
\begin{equation}\label{eq:res}
F_j^{n,k}=\big(\phi\rho(\sum_{i}p_i^{n,k}\phi_i),\phi_j\big)-\big(\phi\rho(\sum_{i}p_i^{n-1}\phi_i),\phi_j\big)+\delta t^n\big(\frac{\kappa}{\mu}\nabla \rho(\sum_{i}p_i^{n,k}\phi_i),\nabla \phi_j\big)-\delta t^n(q_h,\phi_j)=0
\end{equation}
for $j=1,2,\ldots, N_f$.

To linearize the global problem, we should compute the partial derivatives of the
residual equation with respect to the unknown $p_i^{n,k}$
\begin{equation}\label{eq:partial}
J^{n,k}_{ji}:=\frac{\phi\partial F_j^{n,k}}{\partial p_i^{n,k}}=
(\phi\rho(p^{n,k})\phi_i,\phi_j)+\delta t^n(\frac{\kappa}{\mu}\rho(p^{n,k})\nabla\phi_i,\nabla\phi_j)+
\delta t^n(c\frac{\kappa}{\mu}\phi_i\rho(p^{n,k}),\nabla\phi_j),
\end{equation}
which results in a linear system that needs to solve
\begin{equation}\label{eq:linsys}
J^{n,k}\delta p^{n,k}=F^{n,k},
\end{equation}
where the Jacobi matrix $J^{n,k}=[J^{n,k}_{ji}]$, the residual $F^{n,k}=[F_j^{n,k}]$.
Then $p^{n,k+1}=p^{n,k}+\delta p^{n,k}$.

\section{Offline coarse space and coarse problem}\label{sec:offline}

The construction of the spectral coarse space consists of two steps. First, we construct the local snapshot space. Second, we reduce the dimension of the snapshot space by using a carefully defined spectral problem.  We first present the construction  of local snapshot spaces in $\omega_i$.
There are two types of local snapshot spaces.
The first type is

\begin{equation*}
V_1^{i,\text{snap}} = V^h(\omega_i),
\end{equation*}
where $V(\omega_i)$ is the restriction of $V^h$ to $\omega_i$.
Therefore, $V_1^{i,\text{snap}}$ contains all possible fine scale functions defined on $\omega_i$. The second type is the harmonic extension space. More specifically, let $V^h(\partial\omega_i)$ be the restriction of the conforming space $V^h$ to $\partial\omega_i$.
Then we define the fine-grid delta function $\delta_k \in V^h(\partial\omega_i)$ on $\partial\omega_i$ by

\begin{equation*}
\delta_k( {x}_l) =
\begin{cases}
1, \quad & l = k, \\
0, \quad & l \ne k,
\end{cases}
\end{equation*}
where $ \{{x}_l\}$ are all fine grid nodes on $\partial\omega_i$.
Given $\delta_k$, we seek $ {u}_{k}$ by
\begin{equation}\label{eq:cg_snap_har}
\begin{cases}
\begin{aligned}
- \nabla\cdot(\rho(p_0){\kappa} \nabla u_{k}) &=  {0}, &&\text{in} \ \omega_i, \\
{u}_{k} &= \delta_k, &&\text{on} \ \partial\omega_i.
\end{aligned}
\end{cases}
\end{equation}
The linear span of the above harmonic extensions is our second type local snapshot space $V^{i,\text{snap}}_2$.
To simplify the representations, we will use $V^{i,\text{snap}}$ to denote $V^{i,\text{snap}}_1$ or $V^{i,\text{snap}}_2$
when there is no need to distinguish them. Moreover, we write

\begin{equation*}
V^{i,\text{snap}} = \text{span} \{  {\psi}^{i,\text{snap}}_k: k=1,2,\cdots, M^{i,\text{snap}} \},
\end{equation*}
where ${\psi}^{i,\text{snap}}_k$ is the snapshot functions, and
$M^{i,\text{snap}}$ is the number of basis functions in $V^{i,\text{snap}}$.

The dimension of the snapshot space is  too rich and thus expensive for computation.
A spectral problem will be performed to select the dominant modes from
the snapshot space.
Specifically, in each neighborhood $\omega_i$, we consider
\begin{equation}\label{eq:spec-cg}
\nabla\cdot(\rho(p_0)\kappa \nabla \phi)=\lambda \rho(p_0)\tilde{\kappa} \phi,
\end{equation}
where $\tilde{\kappa}={\rho(p_0)\kappa}\sum_{i=1}^{N_S} | \nabla \chi_i |^2,$
$N_S$ is the total number of neighborhoods, $p_0$ is the initial $p$
and $\chi_i$ is
the partition of unity function \cite{pu} for $\omega_i$. The choice of this spectral problem is motivated by analysis.
One choice of a partition of unity function is the coarse grid hat function whose value at the coarse vertex $x_i$ is 1 and 0 at all other coarse vertices. An alternative option is to use the multiscale finite element basis function (cf. \cite{hou1997multiscale}).
We solve the above spectral problem (\ref{eq:spec-cg}) in each coarse neighbourhood $\omega_i$ in the local snapshot space
$V^{i,\text{snap}}$.  The eigenvalues are arranged in increasing order such that $\lambda_1^{\omega_i}\leq \lambda_2^{\omega_i}\leq \cdots \leq\lambda_{L_i}^{\omega_i}\leq \lambda_{L_i+1}^{\omega_i}\leq\cdots\leq \lambda_{J_i}^{\omega_i}$ and the corresponding eigenfunctions are defined by $\phi_{l,k}$, where $\phi_{l,k}$ is the $k$-th component of $\phi_l$.
Then we use the first $L_i$ eigenfunctions to construct the local offline space, which is defined by

\begin{equation*}
{\psi}^{i,\text{off}}_l = \sum_{k=1}^{M^{i,\text{snap}}} \phi_{l,k}  {\psi}^{i,\text{snap}}_k, \quad l=1,2,\cdots, L_i.
\end{equation*}
Note that the function ${\psi}^{i,\text{off}}_l$ is not globally
continuous, therefore we need to multiply it with the partition of unity function $\chi_i$. We define the local offline space as

\begin{equation*}
V^{i,\text{off}}_H = \text{span} \{\chi_i{\psi}^{i,\text{off}}_l: l=1,2,\cdots, L_i \},\end{equation*}
then the offline space can be defined as

\begin{equation*}
V^{\text{off}}_H = \text{span} \{V^{i,\text{off}}_H: i=1,2,\cdots, N_S \}.
\end{equation*}
The Equations \eqref{eq:cg_snap_har} and \eqref{eq:spec-cg} are solved on the fine grid $\mathcal{T}^h$ numerically.
Panel (c)-(d) of Figure \ref{fig:eig} show an example of partition of unity $\chi$,
eigenfunction $\psi$ and  offline basis $\chi\psi$ corresponding to
local permeability field displayed in Panel (a) of Figure \ref{fig:eig}, the distortion of the offline basis due to the strong heterogeneity of
the local permeability field can be
observed.
 From Panel (b) of the Figure \ref{fig:eig}, we can see the inverse of the eigenvalues decreases rapidly which implies the dominance of the first several eigenfunctions.

Given the above space, the discrete formulation reads as follows: Find $p_H\in V^{\text{off}}_H$ such that
\begin{align}
(\frac{\partial}{\partial t}(\phi \rho(p_H)),v)+(\frac{\kappa}{\mu}\rho (p_H)\nabla p_H,\nabla v)=(q,v)\quad \forall v\in V^{\text{off}}_H.\label{eq:coarse-sol}
\end{align}
Note that
\begin{align}
R_{p_h}(v)=(q,v)-(\frac{\partial}{\partial t}(\phi \rho(p_h)),v)-(\frac{\kappa}{\mu}\rho (p_h)\nabla p_h,\nabla v)=0\quad \forall v\in V^{\text{off}}_H\label{eq:Rph}
\end{align}
since $V^{\text{off}}_H\subset V_h$.

\begin{figure}[htb!]
	\centering
	\includegraphics[trim={0cm 1.5cm 0cm 0cm},clip,width=6.0in]{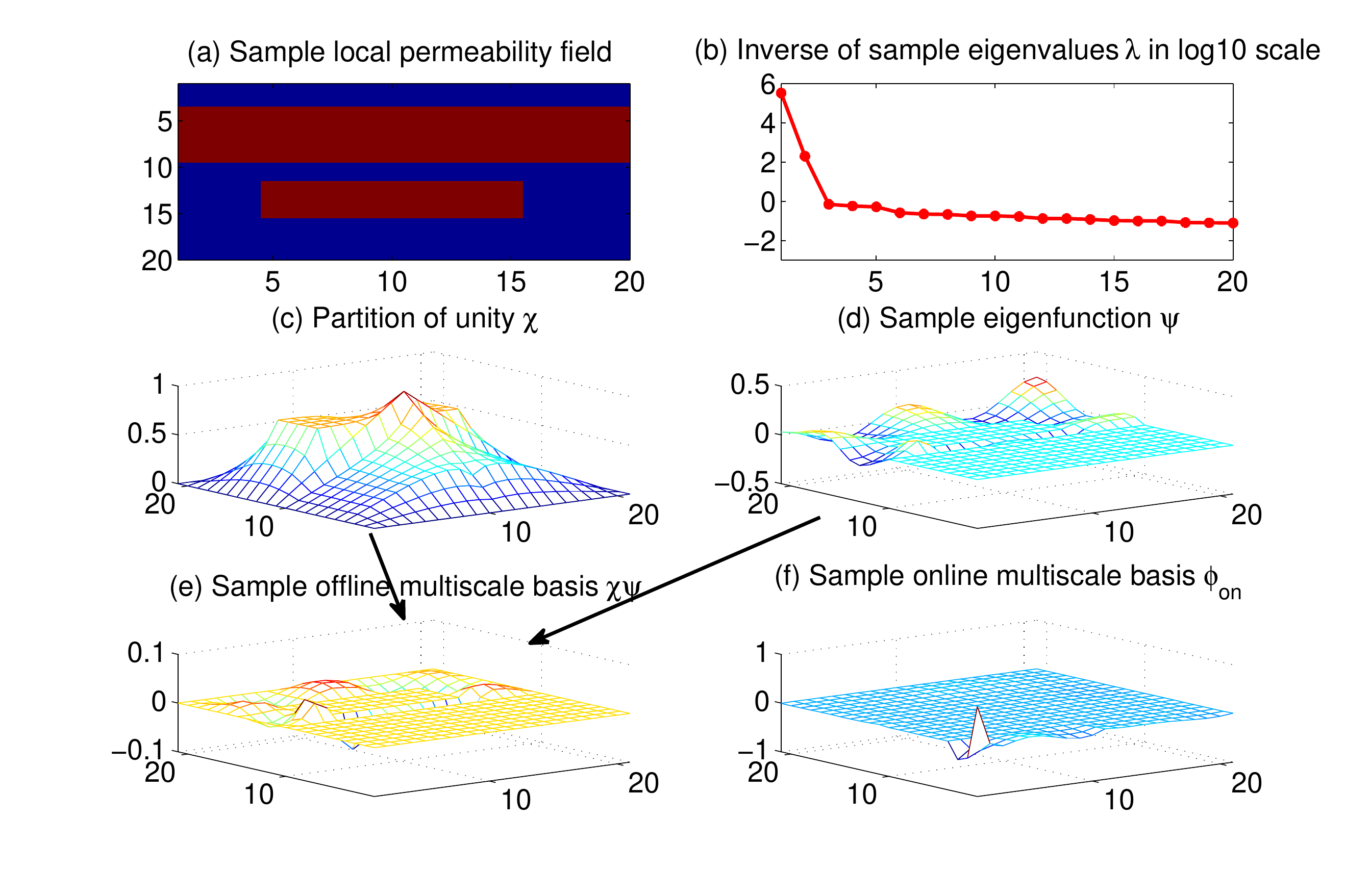}			
	\caption{Panel (a):
		an example of local permeability field. Panel (b): inverse of  20 smallest eigenvalues $\lambda$ in log10 scale. Panel (c): an example of partition of unity $\chi$.
		Panel (d): an example of eigenfunction $\psi$ (e): an example of offline basis $\chi\psi$. Panel (f): an example of online basis $\phi_{\text{on}}$.}
	\label{fig:eig}
\end{figure}

Denote each discrete offline or online multiscale basis
as a column vector $\Phi_i$ and $R = [ \Phi_1,\cdots,  \Phi_L ]$ be the projection matrix that stores all the multiscale
basis functions.
Then the coarse  linear system in Newton's method that requires to be solved is
\begin{equation}\label{eq:coarse}
R^T J^{n,k} R \delta p^{n,k}_H=R^TF^{n,k}.
\end{equation}
The dimension of the matrix $R^T J^{n,k} R$ is much smaller than the matrix
$J^{n,k}$ if only a few multiscale bases are utilized, which implies solving \eqref{eq:coarse} is cheaper than solving \eqref{eq:linsys}.
Once $\delta p^{n,k}_H$ is
obtained one can project it to the space $V_h$ by using the projection matrix $R^T$
via $p^{n,k}_{H,f}=R\delta p^{n,k}_H$ to seek a fine-scale representation of the coarse-grid update.
We summarize the algorithm of using GMsFEM and Newton's method for solving \eqref{eq:model} in Algorithm 1.
The fine-scale reference solution can be obtained similarly.

\begin{table}
	\centering
	\begin{tabular}{l}
		\hline
		Algorithm 1 \quad\quad  GMsFEM with Newton's method for  solving \eqref{eq:model} \\
		\hline
		
		{Superscript $n$ denotes time index and $n$, $k$ Newton index $k$ at time n. }\\
		Given $p^0_{H}$,\quad \{Initial Condition\}\\
		\textbf{for} $n=1$,\ldots,$N_T$ \{Time index\}\quad {\bf do}\\
		\quad $p^{n,0}_{H}=p^{n-1}_{H}$\quad \{Initial guess for Newton iteration\}\\
		\quad {\bf for} $k=0,\ldots,$ {\text{MAX}\_\text{NEWT}\_\text{IT}} \{Newton index\} \quad {\bf do}\\
		\quad \quad $F^{n,k}=[F_j^{n,k}]$\quad {Form Newton residual}\\
		\quad\quad {\bf if}\quad ($||F^{n,k} ||<\text{NEWT}_\text{TOL}$)\quad
		{\bf then break} $k$-loop \quad \{Check nonlinear convergence\}\\
		\quad\quad $J^{n,k}=[J^{n,k}_{ji}]$\quad {Form matrix}\\
		\quad\quad $\delta p^{n,k}_H=(R^T J^{n,k} R)^{-1}(R^TF^{n,k})$\quad \{Solve linear system \}\\
		\quad\quad $p^{n,k+1}_{H}=p^{n,k}_{H}+R\delta p^{n,k}_H$\quad \{increments unknowns\}\\
		\quad {\bf end for}\\
		\quad $p^n_{H}=p^{n,k}_{H}$\\
		{\bf end for}\\
		\hline
	\end{tabular}
	\label{ta:ag1}
\end{table}
\section{Convergence error estimates}\label{sec:convergence}

In this section, we will present the convergence error estimates for the semi-discrete scheme \eqref{eq:coarse-sol}. Specifically, error estimates based on $V_1^{\text{snap}}$ and $V_2^{\text{snap}}$ are both proved. The analysis consists of two main steps. First, we bound the difference between the fine scale solution and GMsFEM solution by the difference between the fine scale solution and the projection of the fine grid solution. Second, we derive the error estimate for the difference between the fine scale solution and its corresponding projection.

To begin, we recall the continuous Gronwall inequality in the following lemma, see \cite{Cannon99}.
\begin{lemma}\label{lemma:continuousGronwall}
(the continuous Gronwall lemma).
Let $G$ be a nonnegative function and
let $y,f,g$ be locally integrable nonnegative on the interval $[t_0,\infty)$. Assume there exists
a constant $C_0\geq 0$ such that
\begin{align*}
y(t)+G(t)\leq C_0+\int_{t_0}^t f(\tau)\;dt+\int_{t_0}^t g(\tau)y(\tau)\;d\tau\quad \forall t\in [t_0,\infty),
\end{align*}
then
\begin{align*}
y(t)+G(t)\leq \Big(C_0+\int_{t_0}^tf(\tau\;d\tau\Big)e^{\int_{t_0}^tg(\tau)\;d\tau}\quad \forall t\in [t_0,\infty).
\end{align*}

\end{lemma}

%
%
%
%
%
%
%
%
%
%
%
%
%
%

\begin{lemma}\label{thm:error1}

Let $ p_h\in V_h$ be the fine scale solution obtained from \eqref{eq:fine}, $p_H\in V^{\textnormal{off}}_H$ be the GMsFEM solution of \eqref{eq:coarse-sol} and $w_H$ be an arbitrary function belonging to $V^{\textnormal{off}}_H$. Then the following error estimate holds
\begin{align*}
&\|(p_h-p_H)(t)\|_{L^2(D)}^2+\int_0^T \|(\frac{\kappa}{\mu})^{1/2}\nabla (p_h-p_H)\|_{L^2(D)}^2\preceq\|(w_H-p_H)(0)\|_{L^2(D)}^2\\
&\;+\int_0^T\Big(\|(\frac{\kappa}{\mu})^{1/2}\nabla(p_h-w_H)\|_{L^2(D)}^2+
\|p_h-w_H\|_{L^2(D)}^2+\|(p_h-w_H)_t\|_{L^2(D)}^2\Big)\;dt \\ &\;+\|(p_h-w_H)(t)\|_{L^2(D)}^2.
\end{align*}

\end{lemma}

\begin{proof}

For any $v\in V^{\text{off}}_H$, we have the following error equation

\begin{align*}
&(\frac{\partial}{\partial t}(\phi \rho(p_h)),v)+(\frac{\kappa}{\mu}\rho (p_h)\nabla p_h,\nabla v)-(\frac{\partial}{\partial t}(\phi \rho(p_H)),v)-(\frac{\kappa}{\mu}\rho (p_H)\nabla p_H,\nabla v)\\
&=(\frac{\partial}{\partial t}(\phi \rho(p_h)),v)+(\frac{\kappa}{\mu}\rho (p_h)\nabla p_h,\nabla v)-(f,v)=-R_{p_h}(v)=0.
\end{align*}
For any $w_H\in V^{\text{off}}_H$, setting $v=w_H-p_H$ in the above equation yields
\begin{align*}
&(\frac{\partial}{\partial t}(\phi \rho(p_h)),w_H-p_H)+(\frac{\kappa}{\mu}\rho (p_h)\nabla p_h,\nabla(w_H-p_H))-(\frac{\partial}{\partial t}(\phi \rho(p_H)),w_H-p_H)\\
&\;-(\frac{\kappa}{\mu}\rho (p_H)\nabla p_H,\nabla (w_H-p_H))
=0,
\end{align*}
which can be rewritten as
\begin{align}
&(\frac{\partial}{\partial t}(\phi (\rho(p_h)-\rho(p_H))),w_H-p_H)+(\frac{\kappa}{\mu}(\rho (p_h)\nabla p_h-\rho (p_H)\nabla p_H),\nabla(w_H-p_H))=0.\label{eq:error1}
\end{align}
Now we will estimate each term on the left hand side of \eqref{eq:error1} separately. First, the first term on the left hand side of \eqref{eq:error1} can be rewritten as
\begin{equation}
\begin{split}
(\phi\frac{\partial}{\partial t}(\rho(p_h)-\rho(p_H)),w_H-p_H)&=(\phi\frac{\partial}{\partial t}(\rho(w_H)-\rho(p_H)),w_H-p_H)\\
&\;+(\phi\frac{\partial}{\partial t}(\rho(p_h)-\rho(w_H)),w_H-p_H),
\end{split}
\label{eq:error-term1}
\end{equation}
where the first term on the right hand side can be estimated as follows by following \cite{Wheeler73,Park05,KimPark07}.
Notice that
\begin{align*}
(\phi \frac{\partial}{\partial t}(\rho(w_H)-\rho(p_H)),w_H-p_H)&=\frac{d}{dt}\int_D \phi \int_0^{w_H-p_H}\rho'(w_H+\xi)\xi d\xi d\bm{x}\\
&\;+\int_D\phi\int_D (\rho'(p_H)-\rho'(w_H))\frac{\partial w_H}{\partial t}(w_H-p_H)\\
&\;-\int_D \phi \int_0^{w_H-p_H}\rho''(w_H+\xi)\frac{\partial w_H}{\partial t}\xi d\xi d\bm{x},
\end{align*}
where we can estimate the last two terms by
\begin{align*}
|\int_D \phi \int_0^{w_H-p_H}\rho''(w_H+\xi)\frac{\partial w_H}{\partial t}\xi d\xi d\bm{x}|&\preceq \|w_H-p_H\|_{L^2(D)}^2,\\
|\int_D\phi\int_D (\rho'(p_H)-\rho'(w_H))\frac{\partial w_H}{\partial t}(w_H-p_H)|&\preceq \|w_H-p_H\|_{L^2(D)}^2.
\end{align*}
Thus we can infer that
\begin{align*}
(\phi \frac{\partial}{\partial t}(\rho(w_H)-\rho(p_H)),w_H-p_H)\geq \frac{d}{dt}\int_D \phi \int_0^{w_H-p_H}\rho'(w_H+\xi)\xi d\xi d\bm{x}-C_0\|w_H-p_H\|_{L^2(D)}^2,
\end{align*}
where $C_0$ is a positive constant independent of the meshsize.

Then, the first term on the right hand side of \eqref{eq:error-term1} can be bounded by
\begin{align}
(\phi\frac{\partial}{\partial t}(\rho(w_H)-\rho(p_H)),w_H-p_H)\geq \frac{d}{dt}\|w_H-p_H\|_{L^2(D)}^2-C_0\|w_H-p_H\|_{L^2(D)}^2,\label{eq:time1}
\end{align}
where we use
\begin{align}
\int_D \phi \int_0^{w_H-p_H}\rho'(w_H+\xi)\xi d\xi d\bm{x}\geq C \|w_H-p_H\|_{L^2(D)}^2.\label{eq:Ihpph}
\end{align}
The second term on the right hand side of \eqref{eq:error-term1} can be bounded by the chain rule and Young's inequality
\begin{align*}
(\phi\frac{\partial}{\partial t}(\rho(p_h)-\rho(w_H)),w_H-p_H)&=(\phi (\rho'(p_h)-\rho'(w_H))\frac{\partial w_H}{\partial t},w_H-p_H)\\
&\;+(\phi \rho'(p_h)(\frac{\partial p_h}{\partial t}-\frac{\partial w_H}{\partial t}),w_H-p_H)\\
&\preceq\|p_h-w_H\|_{L^2(D)}^2+\|(p_h-w_H)_t\|_{L^2(D)}^2
+\|w_H-p_H\|_{L^2(D)}^2.
\end{align*}
It remains to estimate the second term on the left hand side of \eqref{eq:error1}. We have
\begin{align*}
&(\frac{\kappa}{\mu}\rho (p_h)\nabla p_h,\nabla (w_H-p_H))-(\frac{\kappa}{\mu}\rho (p_H)\nabla p_H,\nabla (w_H-p_H))\\
&=(\frac{\kappa}{\mu}\rho (p_H)\nabla (w_H-p_H),\nabla(w_H-p_H))\\
&\;-(\frac{\kappa}{\mu}\rho (p_H)\nabla w_H,\nabla(w_H-p_H))+(\frac{\kappa}{\mu}\rho (p_h)\nabla p_h,\nabla(w_H-p_H)),
\end{align*}
which can be estimated by
\begin{align*}
(\frac{\kappa}{\mu}\rho (p_H)\nabla (w_H-p_H),\nabla (w_H-p_H))&\geq C \|(\frac{\kappa}{\mu})^{1/2}\nabla (w_H-p_H)\|_{L^2(D)}^2,\\
(\frac{\kappa}{\mu}(\rho(p_h)\nabla p_h-\rho (p_H)\nabla w_H),\nabla(w_H-p_H))&=(\frac{\kappa}{\mu}(\rho(p_h)\nabla p_h-\rho (p_H)\nabla p_h),\nabla(w_H-p_H))\\
&\;+(\frac{\kappa}{\mu}(\rho(p_H)\nabla p_h-\rho (p_H)\nabla w_H),\nabla(w_H-p_H))\\
&\leq C \Big(\|(\frac{\kappa}{\mu})^{1/2}\nabla(p_h-w_H)\|_{L^2(D)}^2+\|p_h-w_H\|_{L^2(D)}^2\Big)\\
&\;+\epsilon\|(\frac{\kappa}{\mu})^{1/2}\nabla(w_H-p_H)\|_{L^2(D)}^2,
\end{align*}
where in the last inequality we use the boundedness of $\rho$ and $\rho'$, and $p_h\in W^{1,\infty}(D)$.

Combining the above estimates and taking $\epsilon$ small enough, we can obtain
\begin{align*}
&\frac{d}{dt}\|w_H-p_H\|_{L^2(D)}^2+\|(\frac{\kappa}{\mu})^{1/2}\nabla (p_H-w_H)\|_{L^2(D)}^2\preceq\|(\frac{\kappa}{\mu})^{1/2}\nabla(p_h-w_H)\|_{L^2(D)}^2\\
&\;+\|p_h-w_H\|_{L^2(D)}^2+
\|(p_h-w_H)_t\|_{L^2(D)}^2+\|w_H-p_H\|_{L^2(D)}^2.
\end{align*}
Integrating with respect to time $t$ and using Gronwall lemma (cf. Lemma~\ref{lemma:continuousGronwall}), we can infer that
\begin{align*}
&\|(w_H-p_H)(t)\|_{L^2(D)}^2+\int_0^T \|(\frac{\kappa}{\mu})^{1/2}\nabla (p_H-w_H)\|_{L^2(D)}^2\preceq\|(w_H-p_H)(0)\|_{L^2(D)}^2\\
&\;+\int_0^T\Big(\|(\frac{\kappa}{\mu})^{1/2}\nabla(p_h-w_H)\|_{L^2(D)}^2
+\|(p_h-w_H)(t)\|_{L^2(D)}^2
+\|(p_h-w_H)_t\|_{L^2(D)}^2\Big)\;dt.
\end{align*}
Thus, the triangle inequality implies
\begin{align*}
&\|(p_h-p_H)(t)\|_{L^2(D)}^2+\int_0^T \|(\frac{\kappa}{\mu})^{1/2}\nabla (p_h-p_H)\|_{L^2(D)}^2\preceq\|(w_H-p_H)(0)\|_{L^2(D)}^2\\
&\;+\int_0^T\Big(\|(\frac{\kappa}{\mu})^{1/2}\nabla(p_h-w_H)\|_{L^2(D)}^2
+\|p_h-w_H\|_{L^2(D)}^2\\
&\;+\|(p_h-w_H)_t\|_{L^2(D)}^2\Big)\;dt +\|(p_h-w_H)(t)\|_{L^2(D)}^2.
\end{align*}
Therefore, the proof is completed.

\end{proof}

In order to prove the error estimate, it remains to show the error bound for $p_h-w_H$. Notice that on each coarse neighborhood $\omega_i$, we can express $p_h$ as
\begin{align}
p_h= \sum_{k=1}^{J_i} c_{k,\omega_i}\psi_{k}^{i,\text{off}},\label{eq:ph}
\end{align}
where $c_{k,\omega_i}$ is determined by a $L^2$-type projection.

Since $w_H$ is an arbitrary function belonging to $V_H^{\text{off}}$, we define $w_H$ on each coarse neighborhood $\omega_i$ by
\begin{align}
w_H=\sum_{k=1}^{L_i} c_{k,\omega_i}\psi_{k}^{i,\text{off}},\label{eq:wH}
\end{align}
where $L_i\leq J_i$ is the number of eigenfunctions selected for the coarse neighborhood $\omega_i$.

In the following, we will prove the error bound for $p_h-w_H$, and multiscale basis based on $V_1^{\text{snap}}$ and $V_2^{\text{snap}}$ will be considered. We first prove the convergence error estimate for $p_h-w_H$ based on $V_1^{\text{snap}}$ and the estimate is stated in Lemma~\ref{lemma:V1snap}.

\begin{lemma}\label{lemma:V1snap}

Let $ p_h\in V_h$ be the fine scale solution obtained from \eqref{eq:fine} and $w_H$ be an arbitrary function belonging to $V^{\textnormal{off}}_H$. Then the following error estimate holds
\begin{align*}
\int_0^T\int_{D}\frac{\kappa}{\mu} \nabla (p_h-w_H)^2\;dx&\preceq \int_0^T (\frac{1}{\Lambda_*})^{n+1}\sum_{i=1}^{N_S}\int_{\omega_i} \frac{\kappa}{\mu}\nabla (p_h-w_H)^2\;dx\\
&\;+\int_0^T \Big((\Lambda_*)^n (\frac{1-\Lambda_*^{-n}}{\Lambda_*-1})+1\Big)\sum_{i=1}^{N_S} \int_{\omega_i}(\frac{\kappa}{\mu}|\nabla \chi_i|^2)^{-1} r^2\;dx,
\end{align*}
where $r$ represents the residual and is defined by
\begin{align}
r:=\frac{\partial (\phi \rho(p_h))}{\partial t}-\nabla \cdot (\frac{\kappa}{\mu}\rho(p_h) \nabla p_h)-\frac{\partial (\phi \rho(w_H))}{\partial t}+\nabla \cdot (\frac{\kappa}{\mu}\rho(w_H) \nabla w_H)\label{eq:r}.
\end{align}

\end{lemma}

\begin{proof}

Multiplying both sides of \eqref{eq:r} by $\chi_i^2(p_h-w_H)$ and integrating over $\omega_i$, we can get
\begin{equation}
\begin{split}
\int_{\omega_i} r\chi_i^2(p_h-w_H)&=\int_{\omega_i}\phi \frac{\partial }{\partial t}( \rho(p_h)-\rho(w_H))\chi_i^2(p_h-w_H)\;dx\\
&\;+\int_{\omega_i} \frac{\kappa}{\mu}(\rho(p_h)\nabla p_h-\rho(w_H)\nabla w_H)\cdot\nabla (\chi_i^2 (p_h-w_H))\;dx.
\end{split}
\label{eq:rchi}
\end{equation}
Proceeding analogously to \eqref{eq:time1}, the first term on the right hand side of \eqref{eq:rchi} can be estimated by
\begin{align*}
\int_{\omega_i}\phi \frac{\partial }{\partial t}( \rho(p_h)-\rho(w_H))\chi_i^2(p_h-w_H)\;dx\geq \frac{d}{dt}\|p_h-w_H\|_{L^2(\omega_i)}^2-C_0 \|p_h-w_H\|_{L^2(\omega_i)}^2
\end{align*}
and the second term on the right hand side of \eqref{eq:rchi} can be bounded by
\begin{align*}
&\int_{\omega_i} \frac{\kappa}{\mu}(\rho(p_h)\nabla p_h-\rho(w_H)\nabla w_H)\cdot\nabla (\chi_i^2 (p_h-w_H))\;dx\\
&=\int_{\omega_i} \frac{\kappa}{\mu}(\rho(p_h)\nabla p_h-\rho(w_H)\nabla w_H)2\chi_i\cdot\nabla \chi_i (p_h-w_H)\;dx\\
&\;+\int_{\omega_i} \frac{\kappa}{\mu}(\rho(p_h)\nabla p_h-\rho(w_H)\nabla w_H)\chi_i^2\cdot\nabla (p_h-w_H)\;dx:=R_1+R_2.
\end{align*}
We can estimate $R_1$ by
\begin{align*}
&\int_{\omega_i} \frac{\kappa}{\mu}(\rho(p_h)\nabla p_h-\rho(w_H)\nabla w_H)2\chi_i\cdot\nabla \chi_i (p_h-w_H)\;dx\\
&=\int_{\omega_i} \frac{\kappa}{\mu}(\rho(p_h)\nabla p_h-\rho(p_h)\nabla w_H)2\chi_i\cdot\nabla \chi_i (p_h-w_H)\;dx\\
&\;+\int_{\omega_i} \frac{\kappa}{\mu}(\rho(p_h)\nabla w_H-\rho(w_H)\nabla w_H)2\chi_i\cdot\nabla \chi_i (p_h-w_H)\;dx\\
&\leq \epsilon_0\int_{\omega_i} \frac{\kappa}{\mu}\chi_i^2\nabla (p_h-w_H)^2+\frac{1}{\epsilon_0}\int_{\omega_i}\frac{\kappa}{\mu}|\nabla \chi_i|^2(p_h-w_H)^2\\
&\;+\int_{\omega_i}\frac{\kappa}{\mu}(p_h-w_H)^2\nabla \chi_i^2+\int_{\omega_i}\frac{\kappa}{\mu} (p_h-w_H)^2\chi_i^2\;dx.
\end{align*}
$R_2$ can be rewritten as follows
\begin{equation}
\begin{split}
&\int_{\omega_i} \frac{\kappa}{\mu}(\rho(p_h)\nabla p_h-\rho(w_H)\nabla w_H)\chi_i^2\cdot\nabla ( p_h-w_H)\;dx\\
&=\int_{\omega_i} \frac{\kappa}{\mu}(\rho(p_h)\nabla p_h-\rho(p_h)\nabla w_H)\chi_i^2\cdot\nabla (p_h-w_H)\;dx\\
&\;+\int_{\omega_i} \frac{\kappa}{\mu}(\rho(p_h)\nabla w_H-\rho(w_H)\nabla w_H)\chi_i^2\cdot\nabla (p_h-w_H)\;dx.
\end{split}
\label{eq:phwH}
\end{equation}
By the boundedness of $\rho$, we can estimate \eqref{eq:phwH} by
\begin{align*}
\int_{\omega_i} \frac{\kappa}{\mu}(\rho(p_h)\nabla p_h-\rho(p_h)\nabla w_H)\chi_i^2\cdot\nabla (p_h-w_H)\;dx&\geq C\int_{\omega_i} \frac{\kappa}{\mu}\chi_i^2\nabla (p_h-w_H)^2\;dx,\\
\int_{\omega_i} \frac{\kappa}{\mu}(\rho(p_h)\nabla w_H-\rho(w_H)\nabla w_H)\chi_i^2\cdot\nabla (p_h-w_H)\;dx&\preceq \frac{1}{\epsilon_1}\int_{\omega_i} \frac{\kappa}{\mu} (p_h-w_H)^2+\epsilon_1 \int_{\omega_i}\frac{\kappa}{\mu} \chi_i^2\nabla (p_h-w_H)^2\;dx.
\end{align*}
Combining the above estimates, and taking $\epsilon_0$ and $\epsilon_1$ small enough, we can obtain
\begin{equation}
\begin{split}
\int_{\omega_i}\frac{\kappa}{\mu} \chi_i^2\nabla (p_h-w_H)^2\;dx+\frac{d}{dt}\|p_h-w_H\|_{L^2(\omega_i)}^2&\preceq \int_{\omega_i} (p_h-w_H)^2\;dx+\int_{\omega_i} \frac{\kappa}{\mu} r\chi_i^2(p_h-w_H)\;dx\\
&\;+\int_{\omega_i} \frac{\kappa}{\mu}|\nabla \chi_i|^2(p_h-w_H)^2\;dx.
\end{split}
\label{eq:part1}
\end{equation}
Consequently, we have
\begin{align*}
\int_{\omega_i}\frac{\kappa}{\mu} \nabla (p_h-w_H)^2\;dx+\frac{d}{dt}\|p_h-w_H\|_{L^2(\omega_i)}^2
&\preceq \int_{\omega_i}(p_h-w_H)^2\;dx+\int_{\omega_i} \frac{\kappa}{\mu}|\nabla \chi_i|^2(p_h-w_H)^2\;dx \\
&\;+\int_{\omega_i}(\frac{\kappa}{\mu}|\nabla \chi_i|^2)^{-1}r^2\;dx.
\end{align*}
Integrating with respect to time and appealing to Lemma~\ref{lemma:continuousGronwall} yields
\begin{align*}
\int_0^T\sum_{i=1}^{N_S}\int_{\omega_i}\frac{\kappa}{\mu} \nabla (p_h-w_H)^2\;dx&\preceq\int_0^T\sum_{i=1}^{N_S}\int_{\omega_i} \frac{\kappa}{\mu}|\nabla \chi_i|^2(p_h-w_H)^2\;dx+\int_0^T\sum_{i}\int_{\omega_i}(\frac{\kappa}{\mu}|\nabla \chi_i|^2)^{-1}r^2\;dx\\
&\;+\|(p_h-w_H)(0)\|_{L^2(D)}^2.
\end{align*}
Then we can get the following estimate by using the spectral problem \eqref{eq:spec-cg}
\begin{align*}
\int_{\omega_i} \sum_{j=1}^{N_S} \kappa |\nabla \chi_j|^2 (p_h-w_H)^2\preceq \frac{1}{\lambda_{L_i+1}^{\omega_i}}\int_{\omega_i} \kappa |\nabla (p_h-w_H)|^2.
\end{align*}
Thereby, can can infer from \eqref{eq:part1} and Gronwall's lemma
\begin{align*}
\int_0^T\sum_i \int_{\omega_i} \kappa |\nabla \chi_i|^2 (p_h-w_H)^2&\preceq \int_0^T\sum_i \int_{\omega_i} \sum_j \kappa |\nabla \chi_j|^2 (p_h-w_H)^2\\
&\preceq \int_0^T\sum_{i=1}^{N_S} \frac{1}{\lambda_{L_i+1}^{\omega_i}} \int_{\omega_i}\kappa |\nabla (p_h-w_H)|^2\\
&\preceq\int_0^T \sum_{i=1}^{N_S} \frac{\alpha_{L_i+1}^{\omega_i}}{\lambda_{L_i+1}^{\omega_i}}\int_{\omega_i} \kappa \chi_i^2 |\nabla (p_h-w_H)|^2\\
&\preceq \int_0^T\sum_{i=1}^{N_S} \frac{\alpha_{L_i+1}^{\omega_i}}{\lambda_{L_i+1}^{\omega_i}}\int_{\omega_i} \kappa |\nabla \chi_i|^2(p_h-w_H)^2+\sum_i \frac{\alpha_{L_i+1}^{\omega_i}}{\lambda_{L_i+1}^{\omega_i}}|\int_{\omega_i} r \chi_i^2(p_h-w_H)|\\
&\preceq \frac{1}{\Lambda_*}\int_0^T\Big(\sum_i \int_{\omega_i} \kappa |\nabla \chi_i|^2(p_h-w_H)^2+\sum_{i=1}^{N_S} |\int_{\omega_i} r \chi_i^2 (p_h-w_H)|\Big),
\end{align*}
where $\Lambda_*=\min_{\omega_i} \frac{\lambda_{L_i+1}^{\omega_i}}{\alpha_{L_i+1}^{\omega_i}}$ and $\alpha_{L_i+1}^{\omega_i}=\int_{\omega_i} \kappa\nabla (p_h-w_H)^2\;dx/\int_{\omega_i} \kappa \chi_i^2 \nabla (p_h-w_H)^2\;dx$.

Then we can apply this inequality $m$ times as in \cite{Efendiev11} to get
\begin{align*}
\int_0^T\int_{D}\frac{\kappa}{\mu} \nabla (p_h-w_H)^2\;dx&\preceq \int_0^T (\frac{1}{\Lambda_*})^{m+1}\sum_{i=1}^{N_S}\int_{\omega_i} \frac{\kappa}{\mu}\nabla (p_h-w_H)^2\;dx\\
&\;+\int_0^T \Big((\Lambda_*)^m (\frac{1-\Lambda_*^{-m}}{\Lambda_*-1})+1\Big)\sum_{i=1}^{N_S} \int_{\omega_i}(\frac{\kappa}{\mu}|\nabla \chi_i|^2)^{-1} r^2\;dx+\|(p_h-w_H)(0)\|_{L^2(D)}^2.
\end{align*}

\end{proof}

We assume that there exists a global function $F$ and a bounded constant $C$, $\int_D F^2\;dx\leq C$ such that
\begin{align*}
\int_{\omega_i} (H^{-2}\kappa |\nabla \chi_i|^2)^{-1} r^2\preceq F^2.
\end{align*}
Next, note that
\begin{align*}
\int_{\omega_i} \kappa \nabla (p_h-w_H)^2\;dx \preceq \int_{\omega_i} \kappa |\nabla p_h|^2\;dx.
\end{align*}
Then we have the following convergence error estimate
\begin{align*}
\int_0^T \int_{D} \kappa |\nabla (p_h-p_H)|^2\preceq \int_0^T(\frac{1}{\Lambda_*})^{m+1}\int_D \kappa (\nabla p_h)^2\;dx+\int_0^T\Big((\Lambda_*)^m(\frac{1-\Lambda_*^{-m}}{\Lambda_*-1})+1\Big)H^2 \int_D F^2\;dx.
\end{align*}
If we take $m=-\frac{\text{log}(H))}{\text{log}(\Lambda_*)}$, then we can get
\begin{align*}
\int_0^T \int_{D} \kappa |\nabla (p_h-p_H)|^2\preceq \int_0^T \frac{H}{\Lambda_*}\int_D \Big(\kappa |\nabla p_h|^2\;dx+1\Big).
\end{align*}

Now we will present the error estimate for $p_h-w_H$ based on $V_2^{\text{snap}}$. First, we note that both $p_h$ and $w_H$ are harmonic functions according to the definitions of $p_h$ (cf. \eqref{eq:ph}) and $w_H$ (cf. \eqref{eq:wH}).

\begin{lemma}\label{eq:inverse}
Let $ p_h\in V_h$ be the fine scale solution obtained from \eqref{eq:fine} and $w_H$ be an arbitrary function belonging to $V^{\textnormal{off}}_H$. Then there holds
\begin{align*}
\int_{\omega_i}\chi_i^2 \kappa |\nabla (p_h-w_H)|^2\;dx\preceq H^{-2}\int_{\omega_i} \tilde{\kappa}(p_h-w_H)^2\;dx.
\end{align*}

\end{lemma}

\begin{proof}
It can be proved by proceeding analogously to Lemma~4.12 of \cite{Li19}, which is thus omitted.

\end{proof}

\begin{lemma}\label{thm:error2}

Let $ p_h\in V_h$ be the fine scale solution obtained from \eqref{eq:fine} and $w_H$ be an arbitrary function belonging to $V^{\textnormal{off}}_H$. Then the following error estimate holds
\begin{align*}
\|\kappa^{1/2}\nabla (p_h-w_H)\|_{L^2(D)}^2\preceq \max_{i=1,\cdots,N_S}(H^{-2}(\lambda_{L_i+1}^{\omega_i})^{-1})\|(\rho(p_0)\kappa)^{\frac{1}{2}}\nabla p_h\|_{L^2(D)}^2.
\end{align*}

\end{lemma}

\begin{proof}

Recall that $p_h$ in each coarse neighborhood $\omega_i$ is defined by \eqref{eq:ph}, thereby
$p_h=\sum_{i=1}^{N_S}\chi_i\sum_{k=1}^{J_i} c_{k,\omega_i}\psi_{k}^{i,\text{off}}$ in $D$. We have
\begin{align*}
p_h-w_H = \sum_{i=1}^{N_S}\chi_i\sum_{k=L_i+1}^{J_i} c_{k,\omega_i}\psi_{k}^{i,\text{off}}.
\end{align*}
Employing the properties of partition of unity function (cf. \cite{Melenk96}) and Lemma~\ref{eq:inverse}, we can infer that
\begin{align*}
\|\kappa^{1/2}\nabla (p_h-w_H)\|_{L^2(D)}^2&\leq \sum_{i=1}^{N_S} \int_{\omega_i} \kappa|\nabla (\chi_i (p_h-w_H))|^2\;dx\\
&\leq \sum_{i=1}^{N_S} \int_{\omega_i}\kappa(\nabla \chi_i)^2 (p_h-w_H)^2\;dx+\sum_{i=1}^{N_S} \int_{\omega_i}\kappa\chi_i^2 \nabla (p_h-w_H)|^2\;dx\\
&\leq H^{-2}\sum_{i=1}^{N_S} \int_{\omega_i}\tilde{\kappa} |p_h-w_H|^2\;dx.
\end{align*}
Then we can deduce from the orthogonality of eigenfunctions and the spectral problem \eqref{eq:spec-cg} that
\begin{align*}
\int_{\omega_i} \rho(p_0)\kappa|\nabla p_h|^2\;dx=\sum_{j=1}^{J_i}c_{k,\omega_i}^2\lambda_j^{\omega_i}\int_{\omega_i}(\psi_{j}^{i,\text{off}})^2.
\end{align*}
Thus, we have
\begin{align*}
\int_{\omega_i}\tilde{\kappa} |p_h-w_H|^2\;dx&=\sum_{j=L_i+1}^{J_i} \int_{\omega_i}\tilde{\kappa} c_{j,\omega_i}^2(\psi_{j}^{i,\text{off}})^2\;dx=\sum_{j=L_i+1}^{J_i} (\lambda_j^{\omega_i})^{-1}\lambda_j^{\omega_i}\int_{\omega_i} \tilde{\kappa} c_{j,\omega_i}^2(\psi_{j}^{i,\text{off}})^2\\
&\leq (\lambda_{L_i+1}^{\omega_i})^{-1}\sum_{j=L_i+1}^{J_i} \lambda_j^{\omega_i}\int_{\omega_i} \tilde{\kappa} c_{j,\omega_i}^2(\psi_{j}^{i,\text{off}})^2\;dx\\
&\leq (\lambda_{L_i+1}^{\omega_i})^{-1}\|(\rho(p_0)\kappa)^{\frac{1}{2}}\nabla p_h\|_{L^2(D)}^2,
\end{align*}
thereby
\begin{align*}
\|\kappa^{1/2}\nabla (p_h-w_H)\|_{L^2(D)}^2\preceq\max_{i=1,\cdots,N_S}(H^{-2}(\lambda_{L_i+1}^{\omega_i})^{-1})\|(\rho(p_0)\kappa)^{\frac{1}{2}}\nabla p_h\|_{L^2(D)}^2.
\end{align*}
Hence the proof is completed.
%

%
%
%
%

\end{proof}

Using the spectral problem \eqref{eq:spec-cg}, we can obtain the following error estimates for both $V_1^{\text{snap}}$ and $V_2^{\text{snap}}$ (see also \cite{Galvis10})
\begin{align*}
\|p_h-w_H\|_{L^2(D)}^2\preceq\max_{i=1,\cdots,N_S}(\lambda_{L_i+1}^{\omega_i})^{-1})\|(\rho(p_0)\kappa)^{\frac{1}{2}}\nabla p_h\|_{L^2(D)}^2,\\
\|(p_h-w_H)_t\|_{L^2(D)}^2\preceq \max_{i=1,\cdots,N_S}(\lambda_{L_i+1}^{\omega_i})^{-1})\|(\rho(p_0)\kappa)^{\frac{1}{2}}\nabla (p_h)_t\|_{L^2(D)}^2.
\end{align*}

Now we can state the main result of this section by combining the preceding estimates.
\begin{theorem}
Let $ p_h\in V_h$ be the fine scale solution obtained from \eqref{eq:fine} and $p_H\in V^{\textnormal{off}}_H$ be the GMsFEM solution obtained from \eqref{eq:coarse-sol}. If we take $m=-\frac{\textnormal{log}(H)}{\textnormal{log}(\Lambda_*)}$, then the following error estimate holds for $V_1^{\textnormal{snap}}$
\begin{align*}
&\|(p_h-p_H)(t)\|_{L^2(D)}^2+\int_0^T \|(\frac{\kappa}{\mu})^{1/2}\nabla (p_h-p_H)\|_{L^2(D)}^2\\
&\preceq \int_0^T\Big(\frac{H}{\Lambda_*}
(\|(\rho(p_0)\kappa)^{\frac{1}{2}}\nabla p_h\|_{L^2(D)}^2+1)+\max_{i=1,\cdots,N_S}(\lambda_{L_i+1}^{\omega_i})^{-1}\|(\rho(p_0)\kappa)^{\frac{1}{2}}\nabla (p_h)_t\|_{L^2(D)}^2\Big)\;dt \\ &\;+\max_{i=1,\cdots,N_S}(\lambda_{L_i+1}^{\omega_i})^{-1}\|(\rho(p_0)\kappa)^{\frac{1}{2}}\nabla p_h\|_{L^2(D)}^2.
\end{align*}
In addition, we have the following error estimate for $V_2^{\textnormal{snap}}$
\begin{align*}
&\|(p_h-p_H)(t)\|_{L^2(D)}^2+\int_0^T \|(\frac{\kappa}{\mu})^{1/2}\nabla (p_h-p_H)\|_{L^2(D)}^2\\
&\preceq\int_0^T\Big(\max_{i=1,\cdots,N_S}H^{-2}(\lambda_{L_i+1}^{\omega_i})^{-1}
\|(\rho(p_0)\kappa)^{\frac{1}{2}}\nabla p_h\|_{L^2(D)}^2+\max_{i=1,\cdots,N_S}(\lambda_{L_i+1}^{\omega_i})^{-1}\|(\rho(p_0)\kappa)^{\frac{1}{2}}\nabla (p_h)_t\|_{L^2(D)}^2\Big)\;dt \\ &\;+\max_{i=1,\cdots,N_S}(\lambda_{L_i+1}^{\omega_i})^{-1}\|(\rho(p_0)\kappa)^{\frac{1}{2}}\nabla p_h\|_{L^2(D)}^2.
\end{align*}

\end{theorem}

\section{Residual driven multiscale basis}\label{sec:online}
In practical reservoir applications, the source function may be singular.
In this case, only using permeability dependent local multiscale basis may yield
solutions that are not accurate near the source. One way to remedy this issue is to add residual driven basis functions (also named online basis) to the offline space, the residual driven bases
 are also defined on  coarse neighborhoods and contain global effects due to the global permeability field and source. The basic idea of constructing
residual driven multiscale basis is to solve a local zero Dirichlet boundary condition problem with carefully defined local residual as source, it can be constructed iteratively and starts from using the offline multiscale solution $p^{n,k}_{H}$.  Specifically, given the multiscale solution $p^{n,k}_{H}\in V_H^{\text{off}}$, for each neighborhood $\omega_i$ we define a
 local residual functional
$R_i^{n,k}$ on $V_h^i=V_h(\omega_i)$ as
\begin{equation}\label{eq:residual}
\begin{aligned}
R^{n,k}_i(v)=&\big(\phi\rho(p^{n,k}_{H}),v\big)_{\omega_i}-\big(\phi\rho(p^{n-1}_{H}),v\big)_{\omega_i}+\delta t^n\big(\frac{\kappa}{\mu}\nabla \rho(p^{n,k}_{H}),\nabla v\big)_{\omega_i}-
\\
&\delta t^n(q_h,v)_{\omega_i}-(\phi\rho(p^{n,k}_{H})\phi_i,v)_{\omega_i}-\delta t^n(\frac{\kappa}{\mu}\rho(p^{n,k}_{H})\nabla p^{n,k}_{H},\nabla v)_{\omega_i}.
\end{aligned}
\end{equation}
We  denote the local bilinear form $J^{n,k}_i(p,v)=(\phi\rho(p)p,v)_{\omega_i}+\delta t^n(\frac{\kappa}{\mu}\rho(p)\nabla p,\nabla v)_{\omega_i}$,
then the local residual driven basis $\phi_{\text{on}}^{n,k,i}\in V_i$ is obtained by solving
\begin{equation}
J^{n,k}_i(\phi_{\text{on}}^{n,k,i},v)=R^{n,k}_i(v),\quad \forall v \in V_i
\end{equation}
 with zero Dirichlet boundary condition, therefore, the solution to the above local problem is conforming.
We note this local online basis $\phi_{\text{on}}^{n,k,i}$ is  for the $k-$th
 Newton iteration at time $t^n$.
 Panel (f) of Figure \ref{fig:eig} shows an example of the online basis, which is computed in a local domain that includes the singular source, so we can see this
 online basis is like a singular function which demonstrates that the online basis
 can capture global information of the solution.

In practice, we will not compute online basis at each Newton iteration and each time step, instead
we choose to compute the online basis at the initial Newton iteration and reuse these basis functions at later Newton iterations.
We also note that one can perform the above step of computing online basis iteratively to get multiple online bases.

Now we are in a position to derive the a posteriori error estimates, which motivates the definition shown in \eqref{eq:residual}. For each $i=1,2,\cdots,N_c$, we let $P_i$ be the projection defined by
\begin{align*}
P_i v=\sum_{k=1}^{L_i} \Big(\int_{\omega_i} \tilde{\kappa} \psi_k^{\omega_i,\text{off}}\Big) \psi_k^{i,\text{off}}.
\end{align*}
To ease later analysis, we define the following norm
\begin{align*}
\|v\|_{V_i}^2=\|v\|_{L^2(\omega_i)}^2+\delta t^n\|(\frac{\kappa}{\mu})^{1/2}\nabla v\|_{L^2(\omega_i)}^2.
\end{align*}
The projection $P_i$ satisfies the following stability bound
\begin{align*}
\|\chi_i(P_i v)\|_{V_i}\leq C_{\text{stab}}^{\omega_i} \|v\|_{V_i},
\end{align*}
where the constant $C_{\text{stab}}^{\omega_i}=\max\{1,H^{-1}(\lambda_{L_i+1}^{\omega_i})^{-1/2}\}$. Moreover, the following convergence result holds (cf. \cite{Galvis10,Chung14})
\begin{align}
\|\chi_i(v-P_iv)\|_{V_i}&\leq C_{\text{conv}}^{\omega_i}(\lambda_{L_i+1}^{\omega_i})^{-1/2}\|v\|_{V_i},
\label{eq:convergence-interpolation}
\end{align}
where $C_1$ and $C_{\text{conv}}^{\omega_i}$ are uniform constants. We also define the projection $\Pi$ by $\Pi v=\sum_{i=1}^{N_S} \chi_i(P_iv)$. For the analysis, we let $C_{\text{stab}}=\max_{1\leq i\leq N_S} C_{\text{stab}}^{\omega_i}$ and $C_{\text{conv}}=\max_{1\leq i\leq N_S}C_{\text{conv}}^{\omega_i}$.


In the next theorem, we prove the a posteriori error estimate, which will guide the construction of online basis functions. To simplify the analysis, the proof presented below is based on the nonlinear problem directly without resorting to linearization and our (undisplayed) numerical experiments indicate that the indicator given below behaves similarly to the one shown in \eqref{eq:residual}.
\begin{theorem}
Let $p_h^n\in V_h$ denote the approximation solution of \eqref{eq:femapp} at $t^n$ and $p_H^n\in V_H^{\text{off}}$ denote the GMsFEM solution of the fully discrete scheme of \eqref{eq:coarse-sol} at $t^n$. Then
there exists a positive constant $C$ independent of the meshsize such that
\begin{align*}
&\|p_h^{N_t}-p_H^{N_t}\|_{L^2(D)}^2+\sum_{n=1}^{N_t} \delta t^n\|(\frac{\kappa}{\mu})^{1/2}\nabla (p_h^n-p_H^{n})\|_{L^2(D)}^2\leq C \Big(\sum_{n=1}^{N_t}\sum_{i=1}^{N_S} \|\widetilde{R}_i^n\|_{V_i^*}^2(\lambda_{L_i+1}^{\omega_i})^{-1}+\|p_h^0-p_H^0\|_{L^2(D)}^2\Big),
\end{align*}
where
\begin{align*}
\widetilde{R}_i^n(v)=\delta t^n\int_{\omega_i} q^nv\;dx-\int_{\omega_i}\phi (\rho(p_H^n)-\rho(p_H^{n-1}))v\;dx+\delta t^n\int_{\omega_i}\frac{\kappa}{\mu}\rho (p_H^n)\nabla p_H^n\cdot  \nabla v\;dx
\end{align*}
and the residual norm is defined by
\begin{align*}
\|\widetilde{R}_i^n\|_{V_i^*}=\sup_{v\in L^2(t_n,t_{n+1};H^1_0(\omega_i))}\frac{\widetilde{R}_i^n(v)}{\|v\|_{V_i}}.
\end{align*}

\end{theorem}

\begin{proof}

Recall that the fully discrete scheme for \eqref{eq:fine} is written as follows
\begin{align}
(\frac{\phi (\rho(p_h^n)-\rho(p_h^{n-1})}{\delta t^n},v)+(\frac{\kappa}{\mu}\rho (p_h^n)\nabla p_h^n,\nabla v)=(q^n,v)\quad \forall v\in V_h\label{eq:ph-discrete}
\end{align}
and the fully discrete scheme for \eqref{eq:coarse-sol} by using backward Euler scheme can be written as follows
\begin{align}
(\frac{\phi (\rho(p_H^n)-\rho(p_H^{n-1})}{\delta t^n},v)+(\frac{\kappa}{\mu}\rho (p_H^n)\nabla p_H^n,\nabla v)=(q^n,v)\quad \forall v\in V_H^{\text{off}}.\label{eq:pH-discrete}
\end{align}
From the definition of $\rho$, we can easily verify that
\begin{align*}
(\phi(\rho(p_h^n)-\rho(p_H^{n})),p_h^n-p_H^{n})\geq C \|p_h^n-p_H^{n}\|_{L^2(D)}^2.
\end{align*}
The boundedness of $\rho$ and Young's inequality imply
\begin{align*}
(\phi(\rho(p_h^{n-1})-\rho(p_H^{n-1})),p_h^n-p_H^{n})\leq C\|p_h^{n-1}-p_H^{n-1}\|_{L^2(D)}^2+\epsilon_1\|p_h^n-p_H^{n}\|_{L^2(D)}^2.
\end{align*}
Combining the above two inequalities and taking $\epsilon_1$ small enough yield
\begin{align*}
&(\phi(\rho(p_h^n)-\rho(p_H^{n})),p_h^n-p_H^{n})
-(\phi(\rho(p_h^{n-1})-\rho(p_H^{n-1})),p_h^n-p_H^{n})\\
&\geq C_0(\|p_h^n-p_H^{n}\|_{L^2(D)}^2-\|p_h^{n-1}-p_H^{n-1}\|_{L^2(D)}^2).
\end{align*}
On the other hand, there holds
\begin{align*}
(\frac{\kappa}{\mu}(\rho (p_h^n)\nabla p_h^n-\rho (p_h^n)\nabla p_H^n),\nabla (p_h^n-p_H^{n}))\geq C\|(\frac{\kappa}{\mu})^{1/2}\nabla (p_h^n-p_H^{n})\|_{L^2(D)}^2.
\end{align*}
Therefore, we can obtain
\begin{equation}
\begin{split}
&\|p_h^n-p_H^{n}\|_{L^2(D)}^2-\|p_h^{n-1}-p_H^{n-1}\|_{L^2(D)}^2+\delta t^n\|\nabla (p_h^n-p_H^{n})\|_{L^2(D)}^2\\
& \preceq\Big((\phi(\rho(p_h^n)-\rho(p_H^{n})),p_h^n-p_H^{n})
-(\phi(\rho(p_h^{n-1})-\rho(p_H^{n-1})),p_h^n-p_H^{n})\\
&\;+\delta t^n(\frac{\kappa}{\mu}(\rho (p_h^n)\nabla p_h^n-\rho (p_h^n)\nabla p_H^n),\nabla (p_h^n-p_H^{n}))\Big)\\
&=\Big((\phi(\rho(p_h^n)-\rho(p_H^{n})),p_h^n-p_H^{n})
-(\phi(\rho(p_h^{n-1})-\rho(p_H^{n-1})),p_h^n-p_H^{n})\\
&\;+\delta t^n\Big((\frac{\kappa}{\mu}(\rho (p_h^n)\nabla p_h^n-\rho (p_H^n)\nabla p_H^n),\nabla (p_h^n-p_H^{n}))\\
&\;-(\frac{\kappa}{\mu}(\rho (p_h^n)\nabla p_H^n-\rho (p_H^n)\nabla p_H^n),\nabla (p_h^n-p_H^{n}))\Big)\Big)\\
&=\delta t^n\Big(\Big((q^n,z)-(\frac{\phi (\rho(p_H^n)-\rho(p_H^{n-1}))}{\delta t^n},z)-(\frac{\kappa}{\mu}\rho (p_H^n)\nabla p_H^n,\nabla z)\Big)\\
&\;-(\frac{\kappa}{\mu}(\rho (p_h^n)\nabla p_H^n-\rho (p_H^n)\nabla p_H^n),\nabla z)\Big),
\end{split}
\label{eq:error-right}
\end{equation}
where $z=p_h^n-p_H^{n}$ and we have used \eqref{eq:ph-discrete} in the last equality.

Now we will estimate the right hand side of \eqref{eq:error-right}. First, we have from \eqref{eq:pH-discrete}
\begin{align*}
&\delta t^n\Big((q^n,z)-(\frac{\phi (\rho(p_H^n)-\rho(p_H^{n-1}))}{\delta t^n},z)-(\frac{\kappa}{\mu}\rho (p_H^n)\nabla p_H^n,\nabla z)\Big)\\
&=\delta t^n\Big((f,z-\Pi z)-(\frac{\phi (\rho(p_H^n)-\rho(p_H^{n-1}))}{\delta t^n},z-\Pi z)+(\frac{\kappa}{\mu}\rho (p_H^n)\nabla p_H^n,\nabla (z-\Pi z))\Big)\\
&=\sum_{i=1}^{N_S}\delta t^n\Big( \Big(\int_{\omega_i} f(z-P_i z)\chi_i\;dx-\int_{\omega_i} \frac{\phi (\rho(p_H^n)-\rho(p_H^{n-1}))}{\delta t^n}(z-P_iz)\chi_i\\
&\;+\int_{\omega_i}\frac{\kappa}{\mu}\rho (p_H^n)\nabla p_H^n\cdot  \nabla ((z-P_i z)\chi_i) \Big)\\
&=\sum_{i=1}^{N_S} \widetilde{R}_i^n(\chi_i(z-P_iz)).
\end{align*}
Then an application of \eqref{eq:convergence-interpolation} implies
\begin{align*}
\sum_{i=1}^{N_S} \widetilde{R}_i^n(\chi_i(z-P_iz))&\leq \sum_{i=1}^{N_S} \|\widetilde{R}_i^n\|_{V_i^*}\|\chi_i(z-P_iz)\|_{V_i}\leq C_{\text{conv}}\sum_{i=1}^{N_S} \|\widetilde{R}_i^n\|_{V_i^*}(\lambda_{L_i+1}^{\omega_i})^{-1/2}\|z\|_{V_i}.
\end{align*}
The last term on the right hand side of \eqref{eq:error-right} can be bounded by
\begin{align*}
\delta t^n(\frac{\kappa}{\mu}(\rho (p_h^n)\nabla p_H^n-\rho (p_H^n)\nabla p_H^n),\nabla (p_h^n-p_H^{n}))\preceq\delta t^n\|\rho(p_h^n)-\rho(p_H^n)\|_{L^2(D)}\|(\frac{\kappa}{\mu})^{1/2}\nabla (p_h^n-p_H^{n})\|_{L^2(D)}.
\end{align*}
Combining the above estimates and using Young's inequality, then we can get by summing over $n$
\begin{align*}
\|p_h^{N_t}-p_H^{N_t}\|_{L^2(D)}^2+\sum_{n=1}^{N_t}\delta t^n \|(\frac{\kappa}{\mu})^{1/2}\nabla (p_h^n-p_H^{n})\|_{L^2(D)}^2&\leq C\Big(\sum_{n=1}^{N_t}\sum_{i=1}^{N_S} \|\widetilde{R}_i^n\|_{V_i^*}^2(\lambda_{L_i+1}^{\omega_i})^{-1}
+\sum_{n=1}^{N_t}\delta t^n\|p_h^n-p_H^n\|_{L^2(D)}^2\\
&\;+\|p_h^0-p_H^0\|_{L^2(D)}^2\Big).
\end{align*}
Thereby, an application of the discrete Gronwall lemma yields
\begin{align*}
&\|p_h^{N_t}-p_H^{N_t}\|_{L^2(D)}^2+\sum_{n=1}^{N_t} \delta t^n\|(\frac{\kappa}{\mu})^{1/2}\nabla (p_h^n-p_H^{n})\|_{L^2(D)}^2\leq C \Big(\sum_{n=1}^{N_t}\sum_{i=1}^{N_S} \|\widetilde{R}_i^n\|_{V_i^*}^2(\lambda_{L_i+1}^{\omega_i})^{-1}+\|p_h^0-p_H^0\|_{L^2(D)}^2\Big).
\end{align*}
Therefore, the proof is completed.

%
%
%

\end{proof}

\begin{remark}
To save the computational costs, we will not update the online basis at each Newton iteration and each time step, but infrequent update will be tested.

\end{remark}

\section{Numerical experiments}\label{sec:na}
In this section, we assess the performances of the multiscale method
with some representative examples.
We are particularly interested in evaluating the accuracy and the CPU time reduction of the GMsFEM, we will study the influence of  adding offline and online basis and updating the online basis. To this end, we consider 3 permeability fields shown in Figure \ref{fig:models}.
$K_1$ and $K_2$ are three high-contrast models that are composed by long
channels and inclusions, the values in  blank regions of $\kappa_1$ and $\kappa_2$ are $10^5$ millidarcys, while in other regions, the values are $10^9$ millidarcys.  $K_3$ is the first 30 layers of the famous
SPE10 dataset \cite{spe10} which is widely used in reservoir simulation community to test multiscale methods.

In all  numerical experiments, two types of boundary conditions and sources combinations
are adopted. One is  full zero Neumann boundary condition, then
 the initial pressure field $p_0$ is homogeneous with a value of $2.16\times 10^7$ Pa. There are four vertical injectors in the corners and one sink in the middle of the domain to drive the flow.
Another type of boundary condition we consider is a combination of zero Neumann and nonzero Dirichlet boundary condition \cite{wang2014algebraic}. More specifically,
we impose zero Neumann boundary condition on boundaries of plane $xy$ and $xz$,
and let $p=2.16\times 10^7$ Pa in the first $yz$ plane and $p=2.00\times 10^7$ Pa in the last $yz$ plane for all time instants,
no extra source is imposed and the flow will be driven by the pressure difference,
the initial pressure field  $p_0$ linearly decreases along the $x$ axis and is fixed in the $yz$ plane.

In all numerical tests, we let viscosity $\mu=5$ cP, porosity $\phi=500$, fluid compressibility
$c=1.0\times 10^{-8}$ $\text{1/Pa}$, the reference pressure $p_\text{ref}=2.00\times 10^7$ Pa,
the reference density $\rho_\text{ref}=850$ kg$/$m$^3$.
The grid size and simulation time settings for each test model are shown in Table \ref{ta:setting}.

In all tables shown below, ``Nb'' is the  number of
local multiscale basis , ``$x+y$'' means $x$ offline basis and $y$ online basis. ``$T_\text{basis}$'' is the total CPU time (in seconds)  for computing the offline and online basis and forming the projection matrix $R$.  ``$T_{\text{ass}}$'' records the total CPU time of forming matrix $J^{n,k}$ and vector $f^{n,k}$.
``$T_{\text{solve}}$'' represents the total  CPU time to solve the linear system with direct solver.
The tolerances for Newton iteration is $10^{-6}$.
To quantify the error, we calculate the relative $L^2$ and $H^1$ error between 
GMsFEM solution and the reference fine-grid FEM solution.
All  computations are performed in a server with  Intel(R) Xeon(R) CPU E5-2687W v4 @ 3.00GHz and with Matlab, 32 cores are utilized for computing the eigenfunctions.

\subsection{Test results for $K_1$}
In Table \ref{ta:k1m}, we exhibit the computational performance comparison for $K_1$
with mixed boundary condition. It is clear that adding offline basis can definitely improve the accuracy of the GMsFEM, for example, the relative $L^2$ error are 2.89e-04 and 1.53e-04 with ``4+0'' and ``8+0'' bases, respectively. The dimension of the coarse system increases from 2916 to 5832, note that the dimension of the fine scale system is 274,625, therefore a huge reduction in the degrees of freedom can be achieved. The CPU time for solving the coarse linear system
is less than 10\% percent of fine-scale solve time even if 8 offline bases are utilized.
The CPU time for computing the offline bases can almost be neglected compared with
``$T_\text{ass}$'' and ``$T_\text{solve}$''. For example, it only takes 14.6s
to obtain 8 offline bases and corresponding projection matrix $R$, in contrast, the CPU time for forming matrix and solve linear system are 124.2 seconds and 194.1 seconds respectively. Note that only 32 cores are used here, if more cores are available the offline time may be further reduced.
Another thing we want to mention is that ``$T_\text{ass}$'' is almost the same for each case, this is because all assembling are
performed in fine-scale, it is possible to apply discrete empirical interpolation  method (DEIM\cite{chaturantabut2010nonlinear}) to reduce this computational cost.

We then investigate the effects of including the online bases. It can be observed that
if only using the initial online basis which means no update of the online basis is applied, the initial online basis shows no obvious improvement compared to the offline basis.
For example, the $L^2$ for the case ``4+0'' bases is 3.09e-04 while this value is
3.67e-04 if ``3+1'' bases are utilized. Using multiple online bases is more useful if we compare the errors of the cases ``4+2'' and ``6+0'' bases, but no significant improvement can be observed. However, updating the online basis can obviously yield more accurate solution. We can see that the $L^2$ error of using
``4+1'' bases with 3 updates  is only 1.11e-04, which is about one half of the corresponding error of the case ``Nb'' that is $6+0$.
The pressure profiles  obtained with the FEM and GMsFEM are displayed in
Figure \ref{fig:k1_mixedbd}, the dynamical behaviors of the distorted pressure field can be observed and
the  GMsFEM solution can capture almost all the details of the FEM solution.

The computational comparison results for the full Neumann boundary condition case
are shown in Table \ref{ta:k1n}. More Newton iterations are needed and thus the
``$T_\text{ass}$'' and ``$T_\text{solve}$'' are larger than the case of mixed
boundary condition. Again, more offline bases imply more accurate
GMsFEM solution, however the improvement is not obvious. As we can see the
$L^2$ error only decreases from 7.62e-3 to 5.45e-3 if the number of offline basis doubles from 4. Another easy noticeable difference between the full Neumann
 boundary  condition and mixed boundary condition is the online basis is more effective in reducing the error in the former case. For example, the $L^2$ error is
 1.33e-03 in the case of ``Nb'' that is ``5+1'' and no update is employed, in contrast,
 if only 6 offline basis is used the $L^2$ error is 6.22e-03. So with only one online basis, significant improvement is obtained at the expensive of slightly increasing ``$T_\text{basis}$''.
 Tremendous $H^1$ error reduction can also be observed and thus confirms the
 powerful advantages of online basis over offline basis for singular source problem.
However, more online basis fails to provide further tremendous improvements by comparing the error of the case ``3+2'' and ``4+1''. This is because the online basis are computed based on the initial solution and permeability fields, which change in the following time steps.  Updating the online basis based on the residual \eqref{eq:residual} during the time marching
also leads to more accurate GMsFEM solution as expected.
We show the pressure profiles comparison with singular source and
zero Neumann boundary condition in Figure \ref{fig:k1_neu}, we can see the flow transports from the injector to the sink. It is hard to find any difference between the GMsFEM solution and the reference solution, which indicates that our GMsFEM can yield very accurate solution.

\subsection{Test results for $K_2$}
The computational performance comparison   results with mixed boundary condition and full zero Neumann boundary condition for $K_2$ are displayed in Table \ref{ta:k2m} and \ref{ta:k2n}, respectively. Although the channels and inclusions in $K_2$ are
larger than $K_1$, however,
the test results for $K_2$ are quite similar as $K_1$. Specifically,
enriching the offline space can generate more accurate solution whatever the boundary condition is imposed. Besides, using only initial online basis will not accelerate the convergence of the GMsFEM too much for the mixed boundary condition case, updating the online basis will help. In addition, if  full zero Neumann boundary condition
is imposed, the error of GMsFEM solution with only offline basis is large, which can be alleviated a lot by adding the online basis. For example, the $L^2$ errors
are 1.39e-02 and 4.04e-03 if ``6+0'' and ``5+1'' bases are utilized, respectively.
We show the pressure profiles comparison in Figure \ref{fig:k2_mixedbd} and \ref{fig:k2_neu}. Again, we can observe distorted pressure fields due to the strong
heterogeneity of the permeability field, the GMsFEM can still nevertheless provide an accurate approximation to the fine-scale solution.

\subsection{Test results for $K_3$}
We summarized the test results for $K_3$ with two types of boundary conditions in
Table \ref{ta:k3m} and \ref{ta:k3n}.
For this test model, even for the mixed boundary condition case, the online basis shows higher efficiency  than the offline basis especially for the $H^1$ error, which can be verified by comparing the errors for cases ``3+1'' and ``8+0'' and other
scenarios. If full zero Neumann boundary condition is imposed, the online basis
shows powerful ability in reducing the large error caused by the singular source.
Moreover, tremendous CPU savings (``$T_\text{solve}$'') can also be observed.
By comparing the simulation results for all these three models, we can
see the convergence of online GMsFEM is almost independent of the media geometry.
The comparisons of the pressure profiles are exhibited in Figure \ref{fig:spe3d_mixedbd} and \ref{fig:spe3d_neubd}, which again demonstrates the GMsFEM is capable of
generating an accurate solution with multiscale behavior.

Finally, we summarize the major observations:
\begin{itemize}
	\item Adding offline basis or online basis can improve the accuracy
	 of the GMsFEM,
	 \item One or two online bases are enough to significantly reduce the possible large error of the offline GMsFEM if singular source is imposed,
	 \item Update the online basis can yield more accurate coarse-grid solution,
	 \item The GMsFEM can provide accurate solution with huge computational cost savings.
\end{itemize}
\begin{table}
	\centering \begin{tabular}{|c|c|c|c|c|c|c|c|c|}\hline
		Model & Fine grid &$h$ & $H$ &$\delta t$  & $T$\tabularnewline\hline
		$K_1$&$64^3$& 20  meters &$8h$  & 7 days& $20 \delta t$   \tabularnewline\hline
		$K_2$&$64^3$&   20 meters &$8h$   &7 days&$20 \delta t$   \tabularnewline\hline
		$K_3$&$220\times 30\times 80$& 20  meters  & $10h$ & 1 day& $20 \delta t$    \tabularnewline\hline		
	\end{tabular}
	\caption{Parameter settings in the performance tests, ``Fine Grid'' is the resolution of the permeability field, $h$ is the fine grid size, $H$ is the coarse grid size $\delta t$ is the time step, $T$ the total simulation time. }
	\label{ta:setting}
\end{table}

\begin{figure}
	\centering
	\subfigure[$K_1$]{\includegraphics[trim={0cm 0cm 0cm 0cm},clip,width=3in]{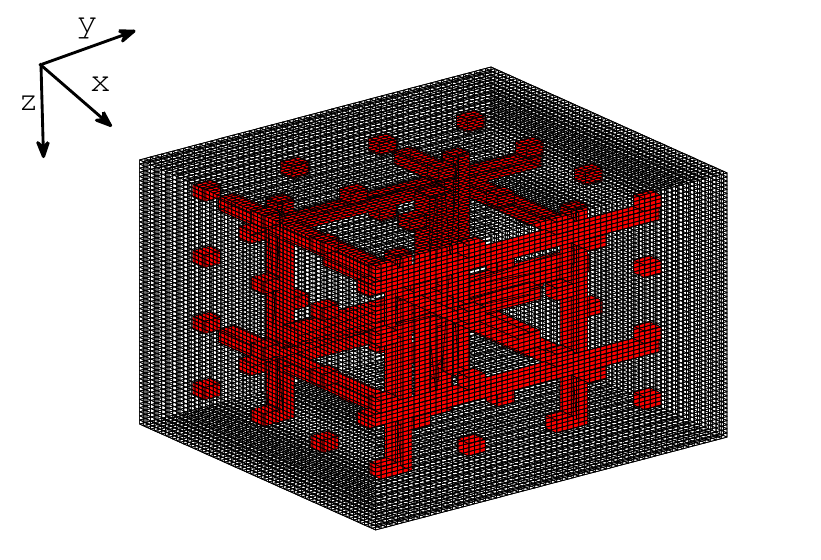}}
	\subfigure[$K_2$]{\includegraphics[trim={0cm 0cm 0cm 0cm},clip,width=2.7in]{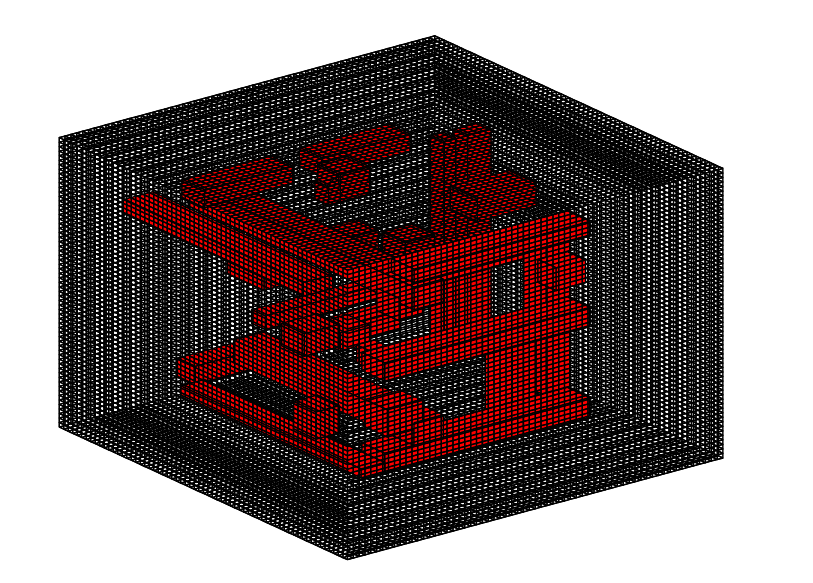}}	
	\subfigure[$K_3$ in log10 scale]{	\includegraphics[trim={1cm 6cm 1cm 6cm},clip,width=3in]{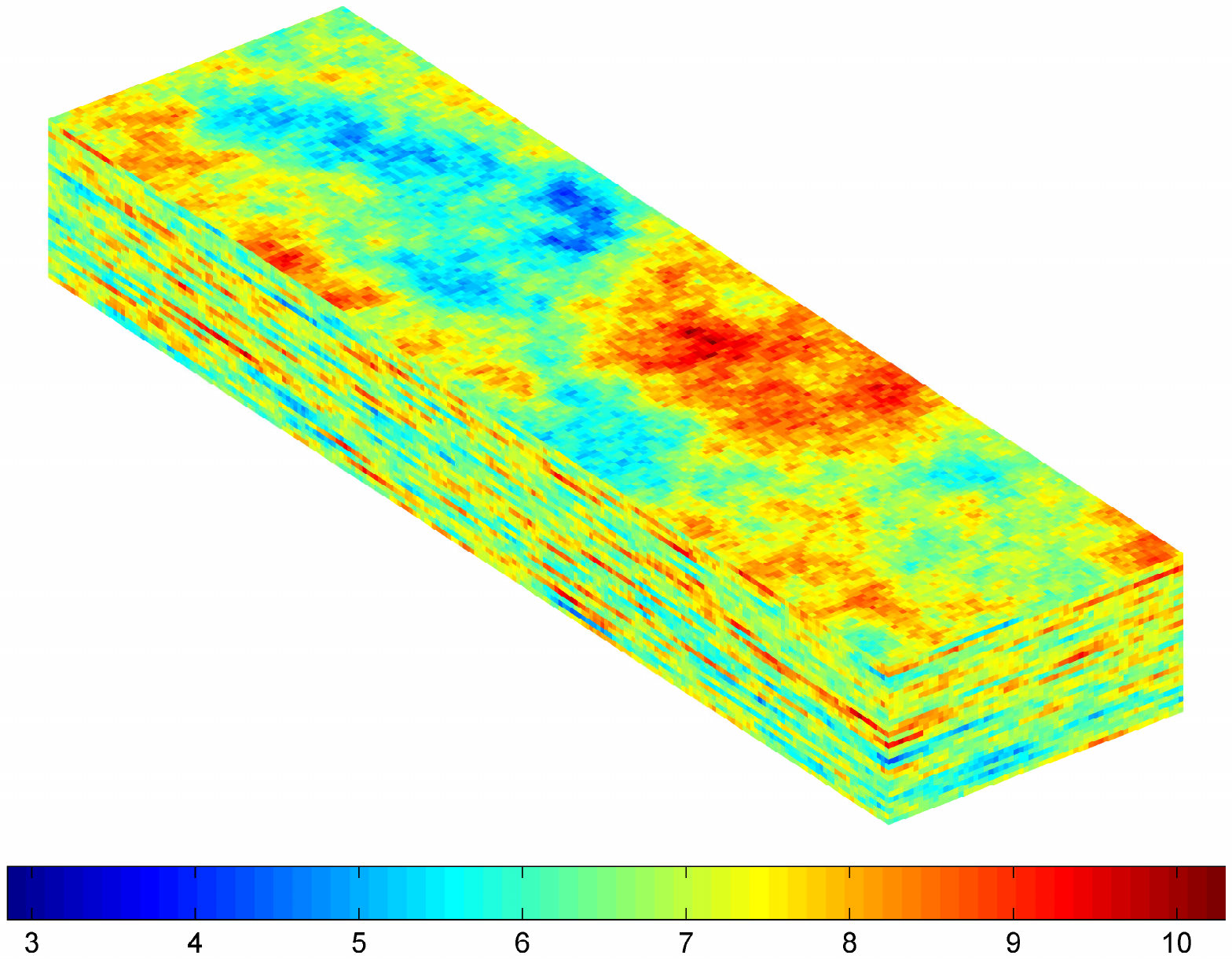}}
	\caption{Test permeability fields $K$. }
	\label{fig:models}
\end{figure}

\begin{figure}[!htb]
	\centering
	\subfigure[Reference solution at Day 35]{\includegraphics[trim={1cm 6cm 1cm 6cm},clip,width=3in]{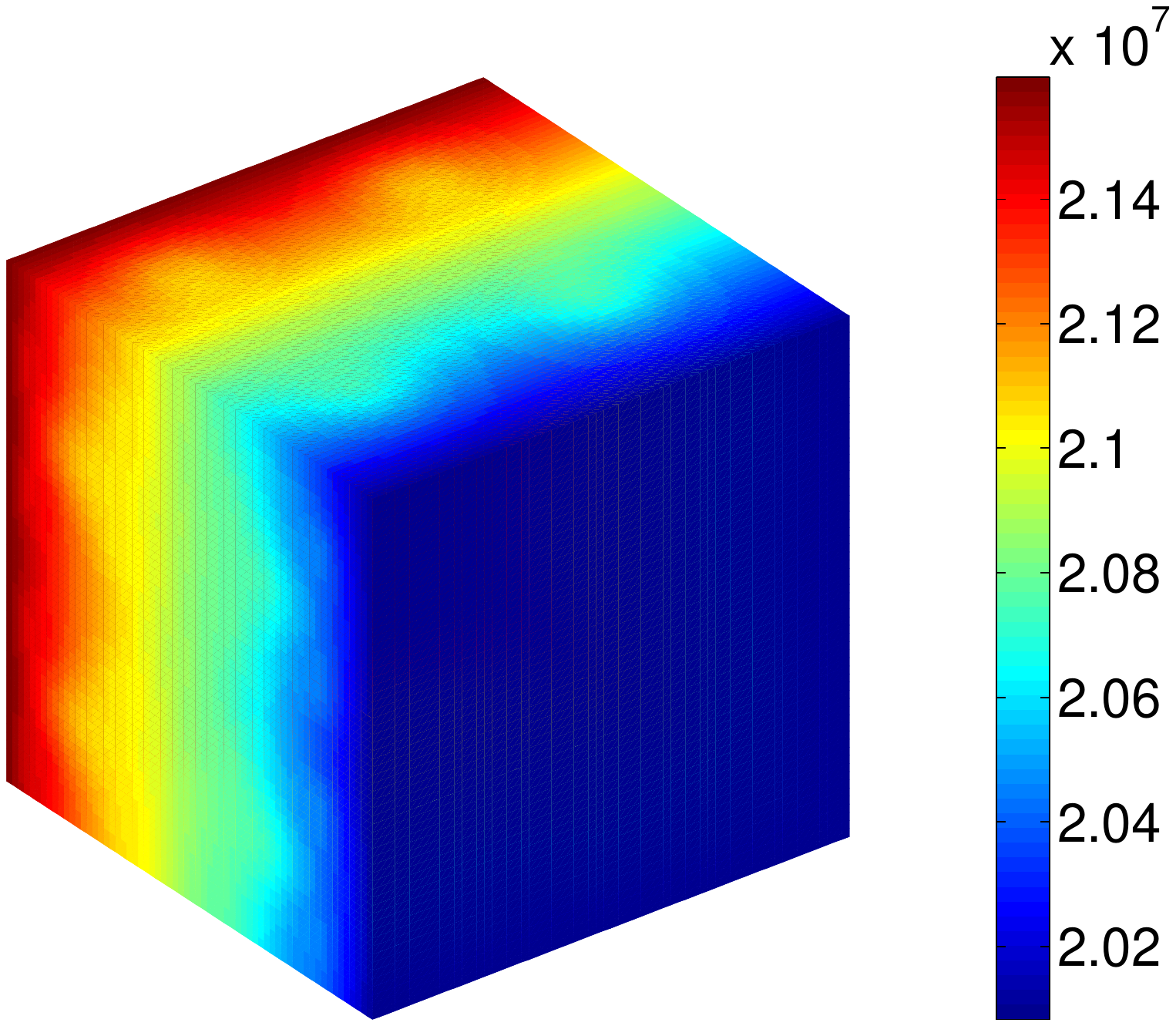}}
	\subfigure[GMsFEM solution with 5+1 bases at Day 35]{\includegraphics[trim={1cm 6cm 1cm 6cm},clip,width=3in]{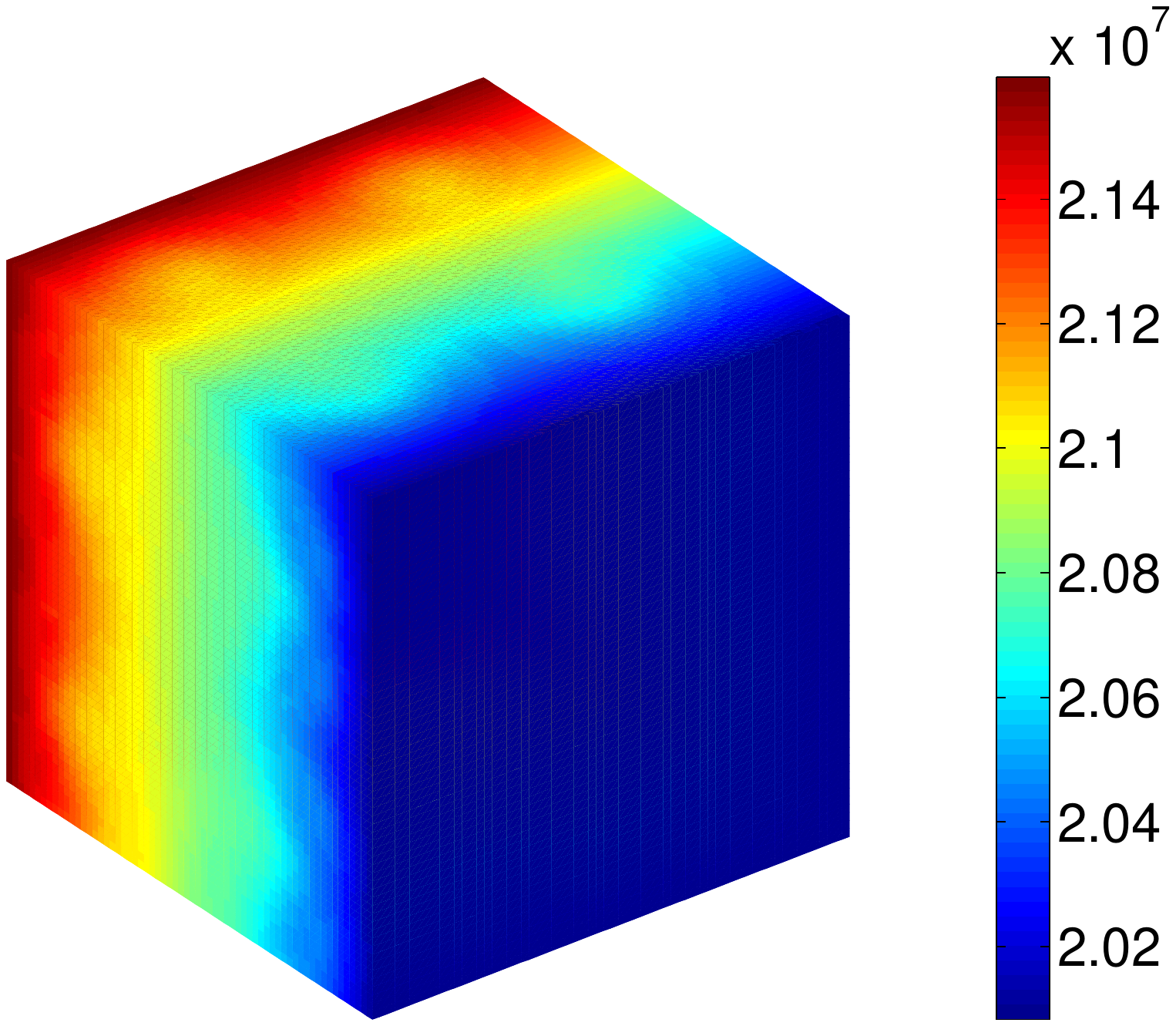}}	
	\subfigure[Reference solution at Day 105]{\includegraphics[trim={1cm 6cm 1cm 6cm},clip,width=3in]{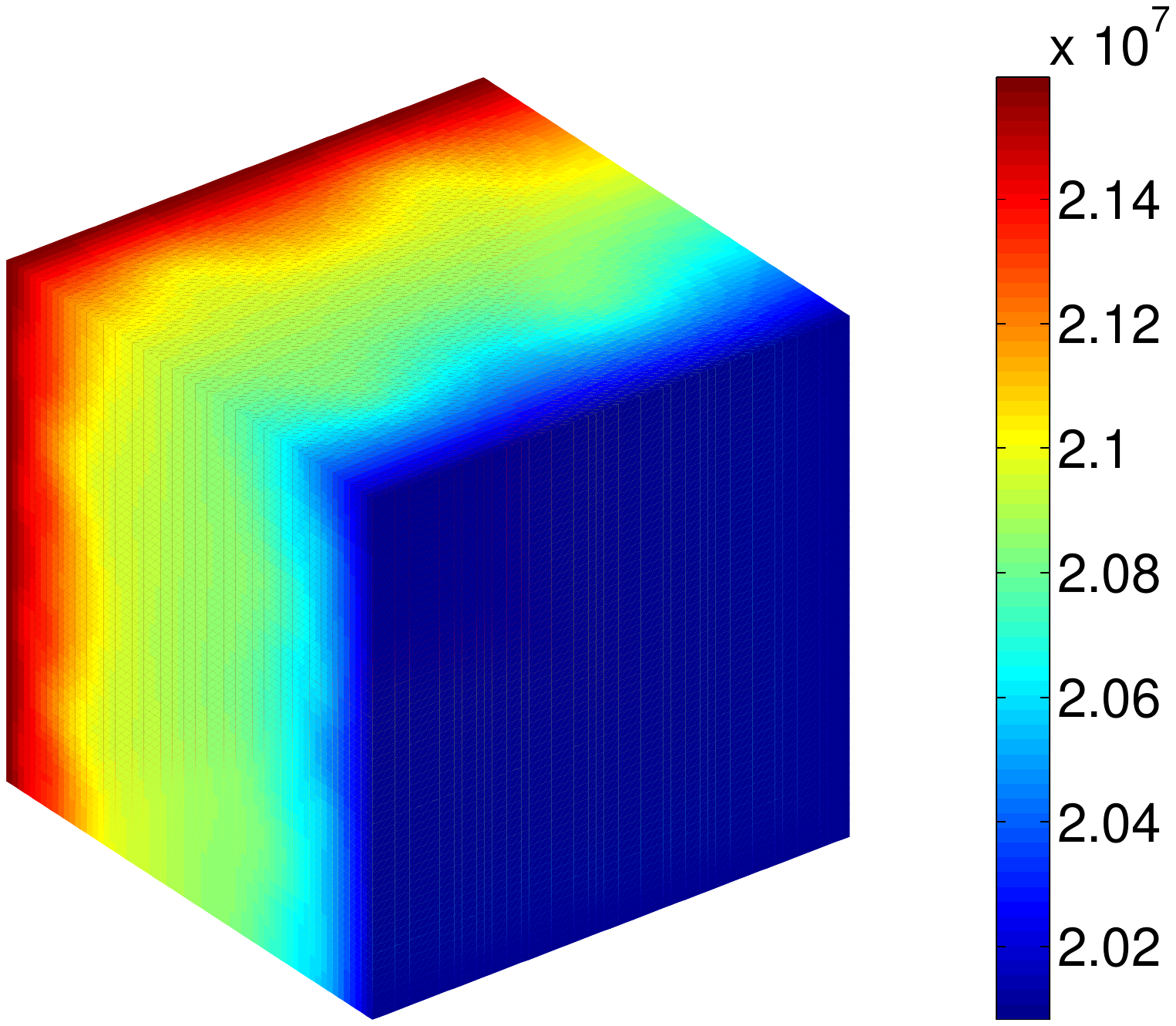}}	
	\subfigure[GMsFEM solution with 5+1 bases at Day 105]{\includegraphics[trim={1cm 6cm 1cm 6cm},clip,width=3in]{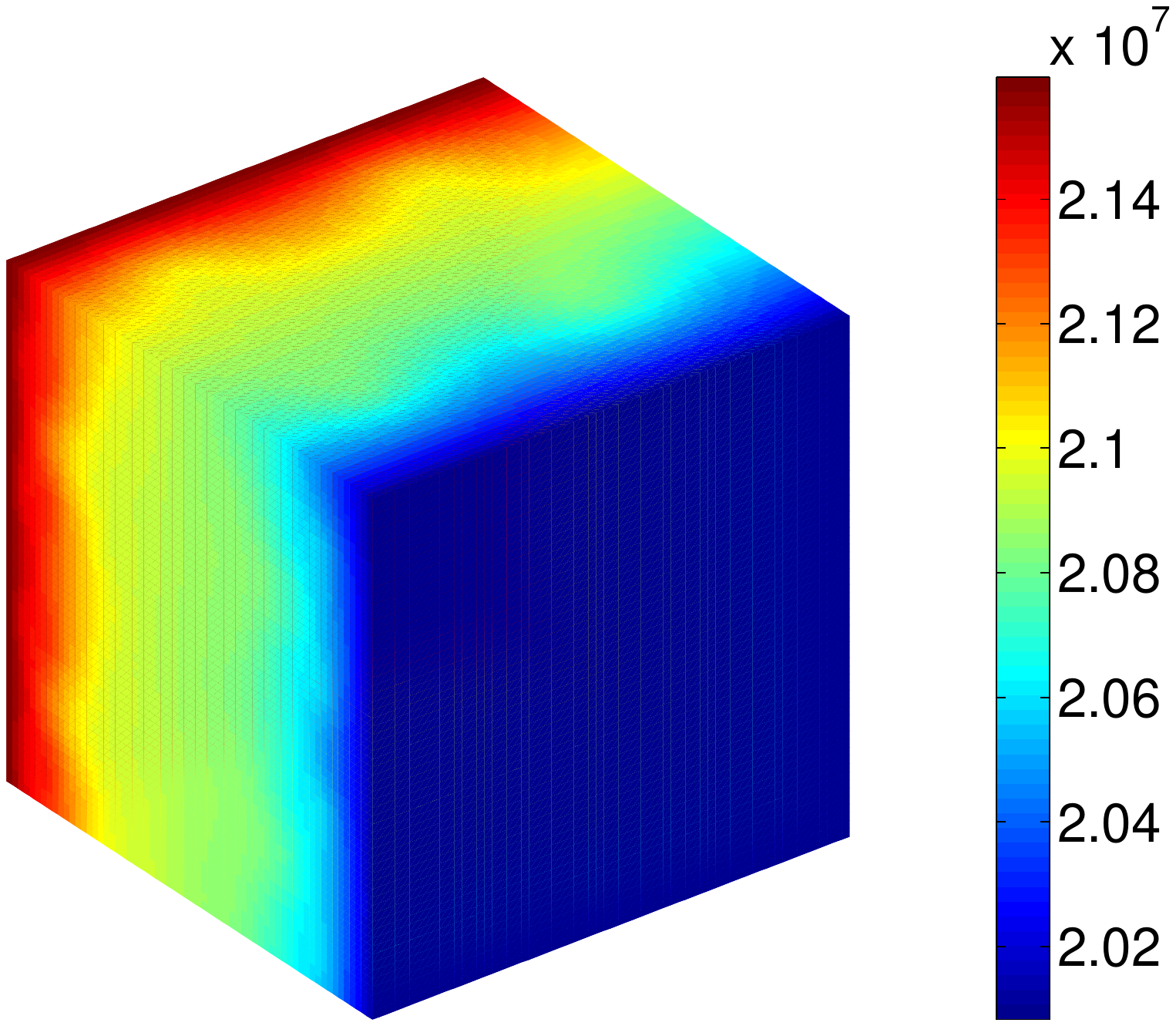}}
	\caption{Fine-scale reference solution and GMsFEM solution with 5+1 bases (3 updates) at different time instants, $K_1$, mixed boundary condition.}
	\label{fig:k1_mixedbd}
\end{figure}

\begin{figure}[!htb]
	\centering
	\subfigure[Reference solution at Day 35]{\includegraphics[trim={1cm 6cm 1cm 6cm},clip,width=3in]{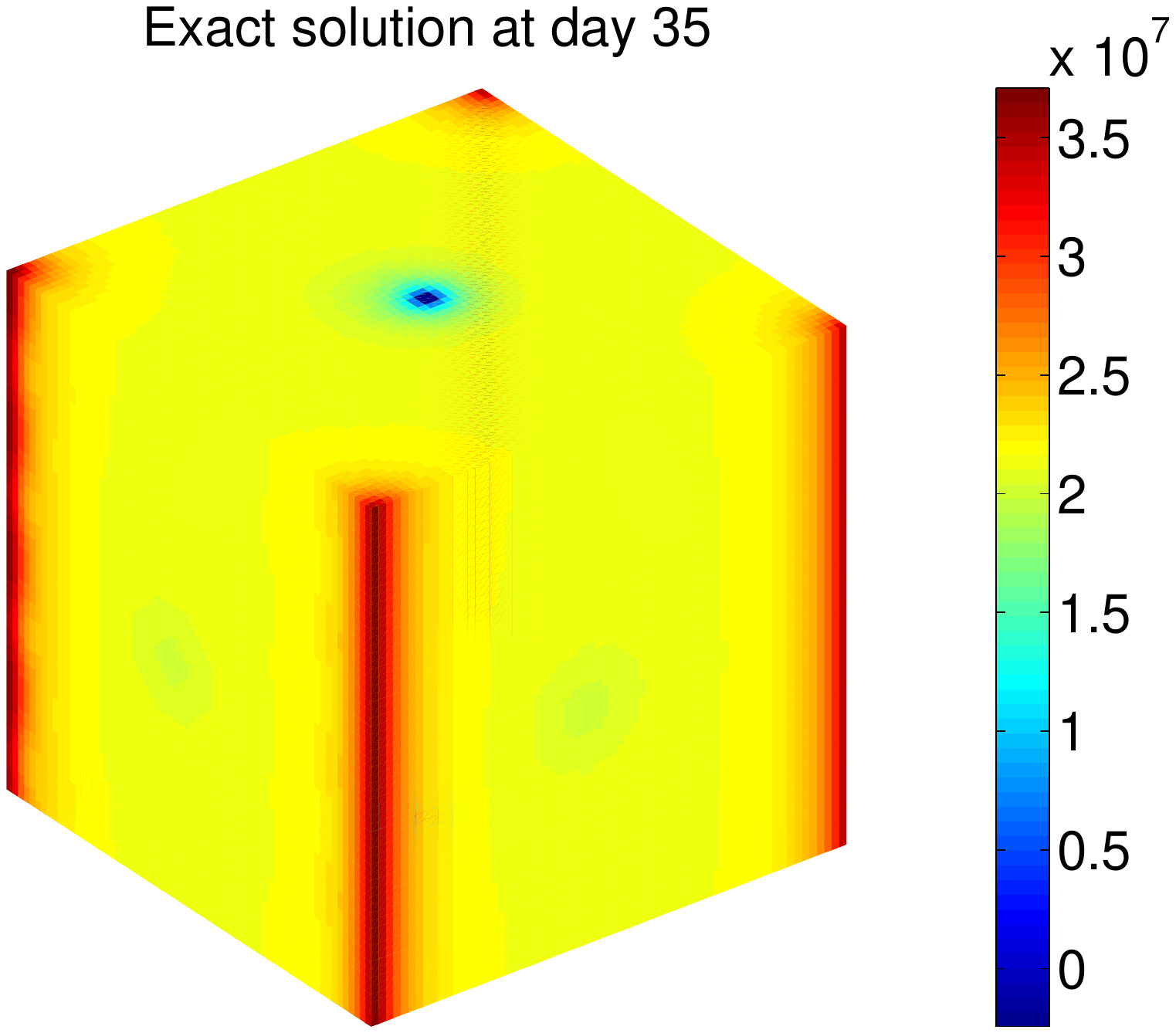}}
	\subfigure[GMsFEM solution with 5+1 bases at Day 35]{\includegraphics[trim={1cm 6cm 1cm 6cm},clip,width=3in]{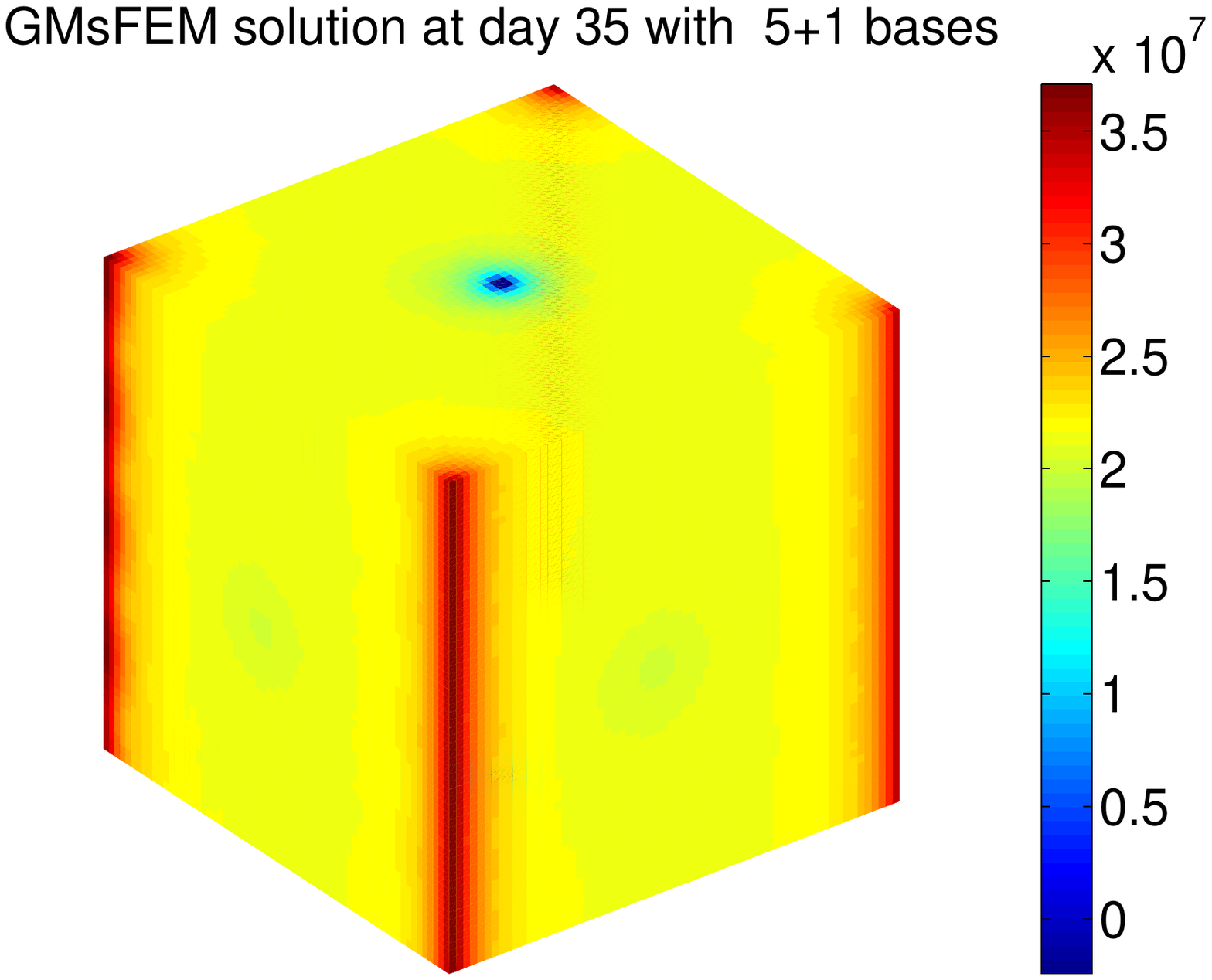}}	
	\subfigure[Reference solution at Day 105]{\includegraphics[trim={1cm 6cm 1cm 6cm},clip,width=3in]{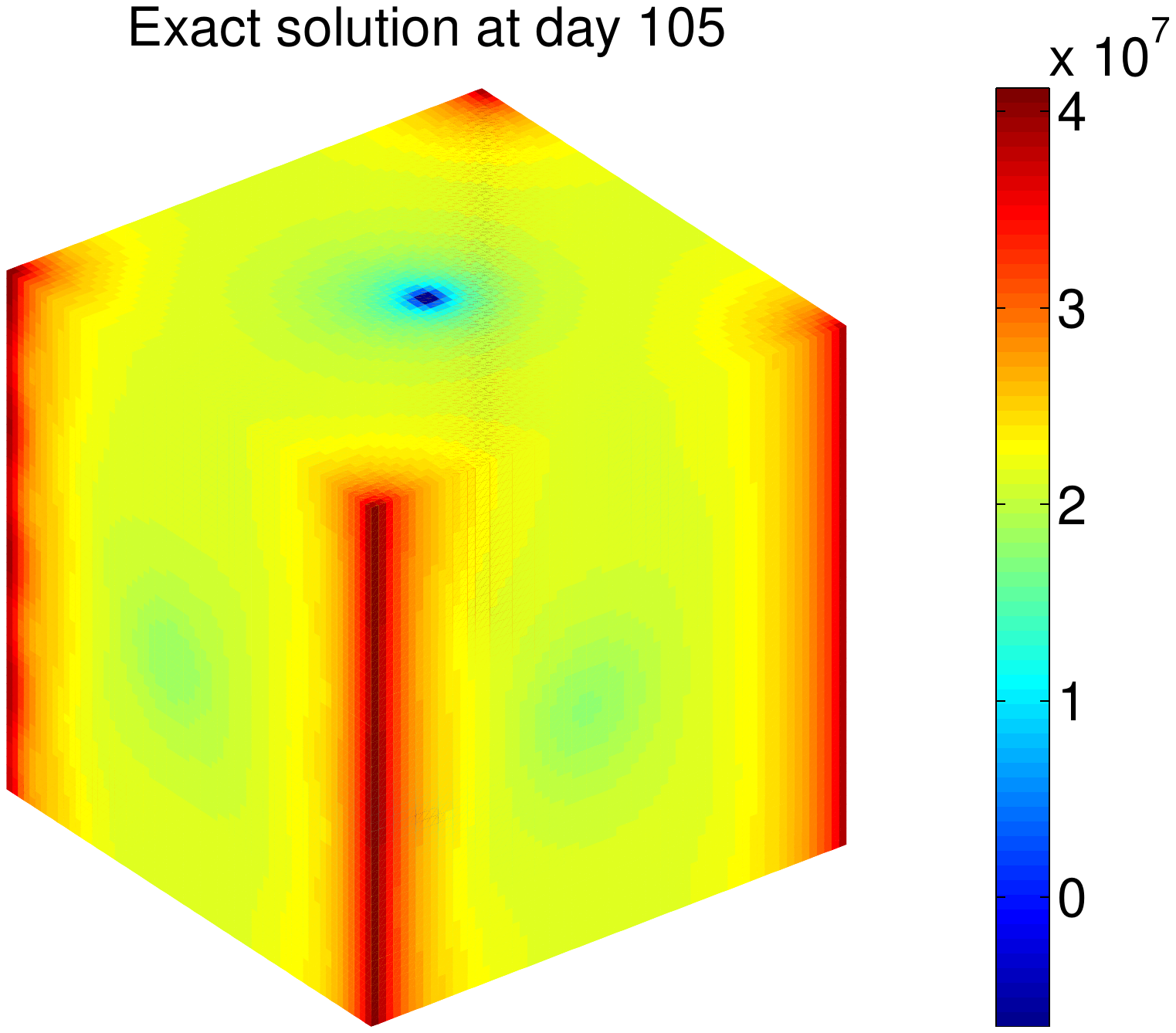}}	
	\subfigure[GMsFEM solution with 5+1 bases at Day 105]{\includegraphics[trim={1cm 6cm 1cm 6cm},clip,width=3in]{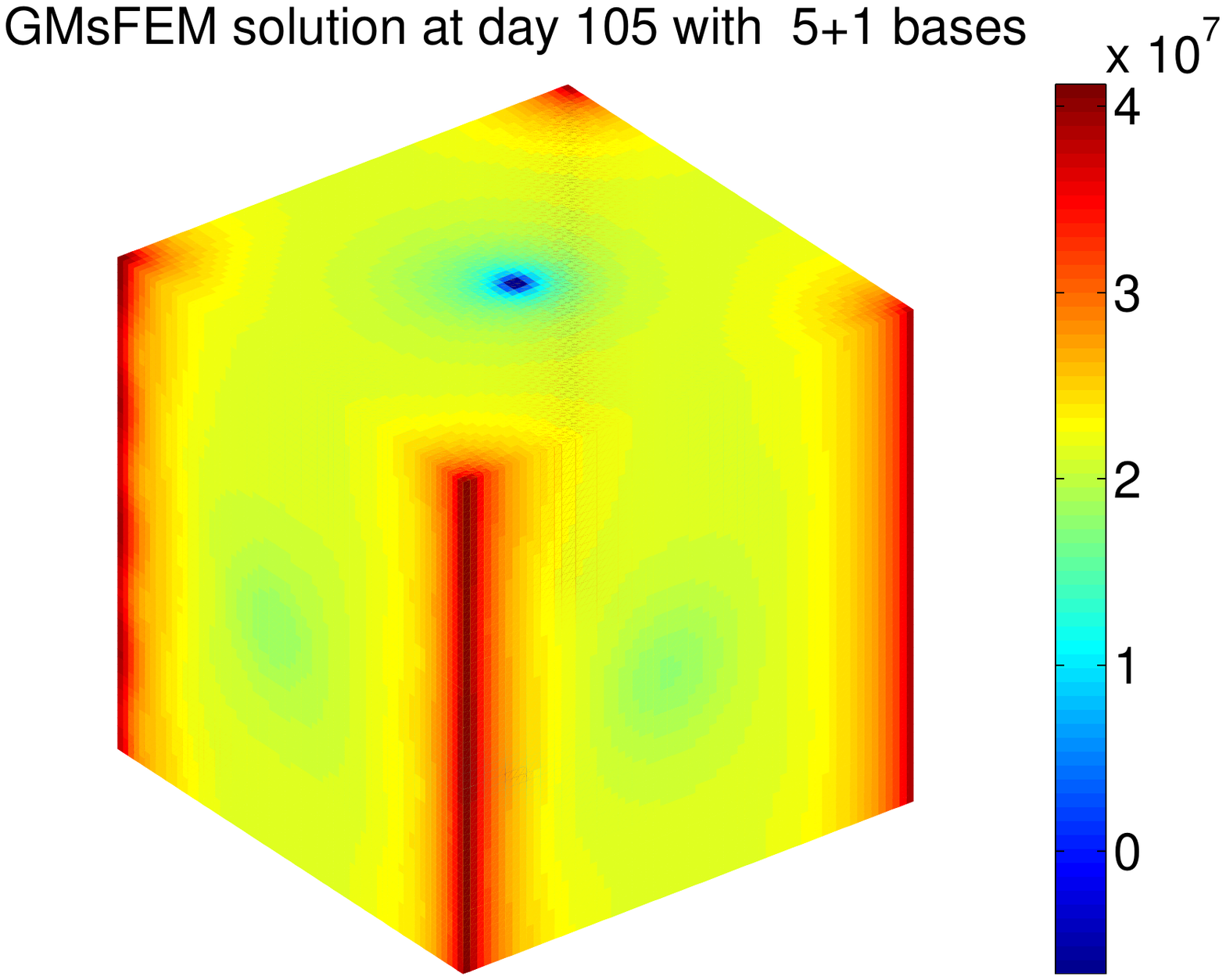}}		
	\caption{Fine-scale reference solution and GMsFEM solution with 5+1 bases (3 updates) at different time instants, $K_1$, full zero Neumann boundary condition.}
	\label{fig:k1_neu}
\end{figure}

\begin{figure}[!htb]
	\centering
	\subfigure[Reference solution at Day 35]{\includegraphics[trim={1cm 6cm 1cm 6cm},clip,width=3in]{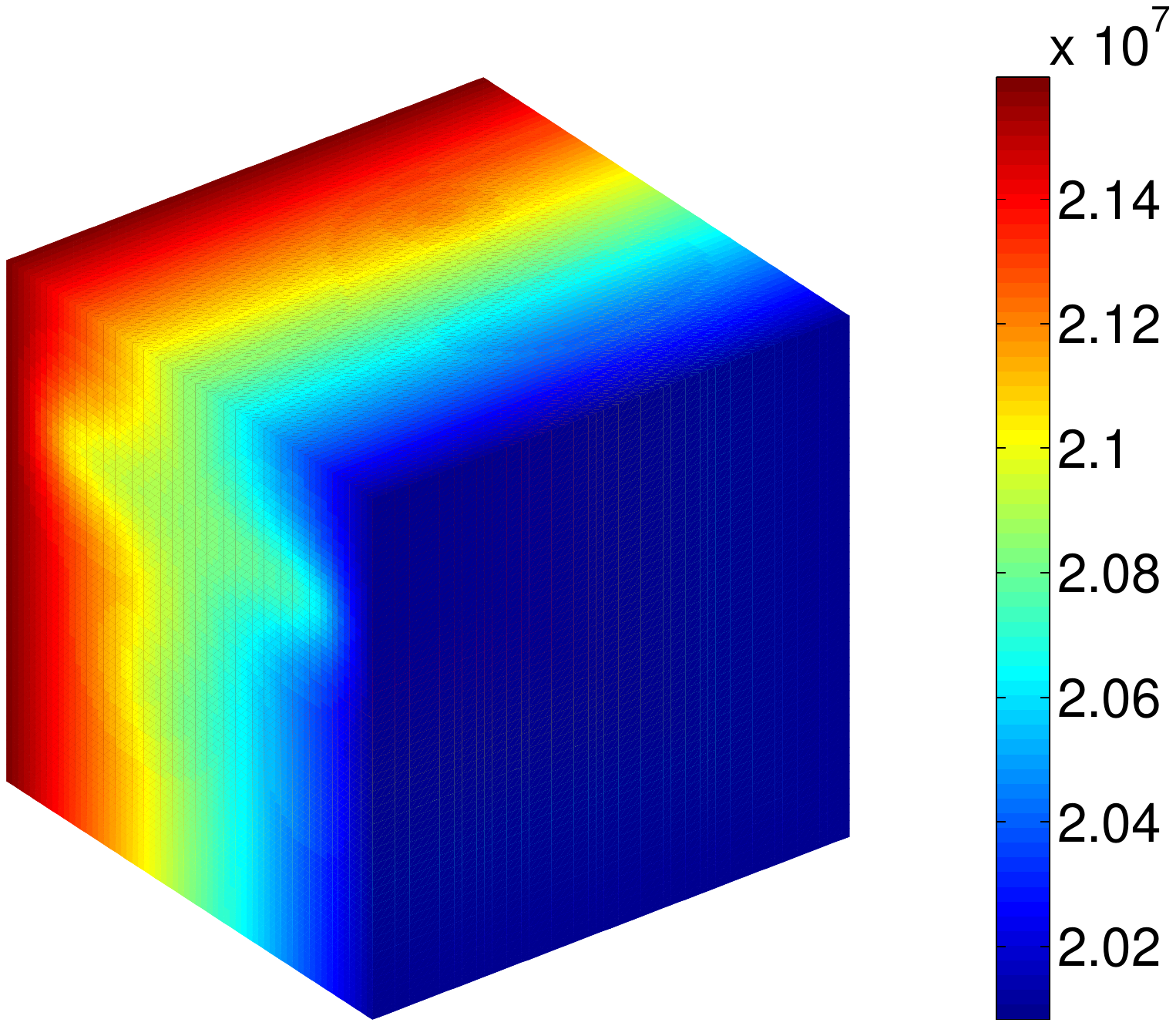}}
	\subfigure[GMsFEM solution with 5+1 bases at Day 35]{\includegraphics[trim={1cm 6cm 1cm 6cm},clip,width=3in]{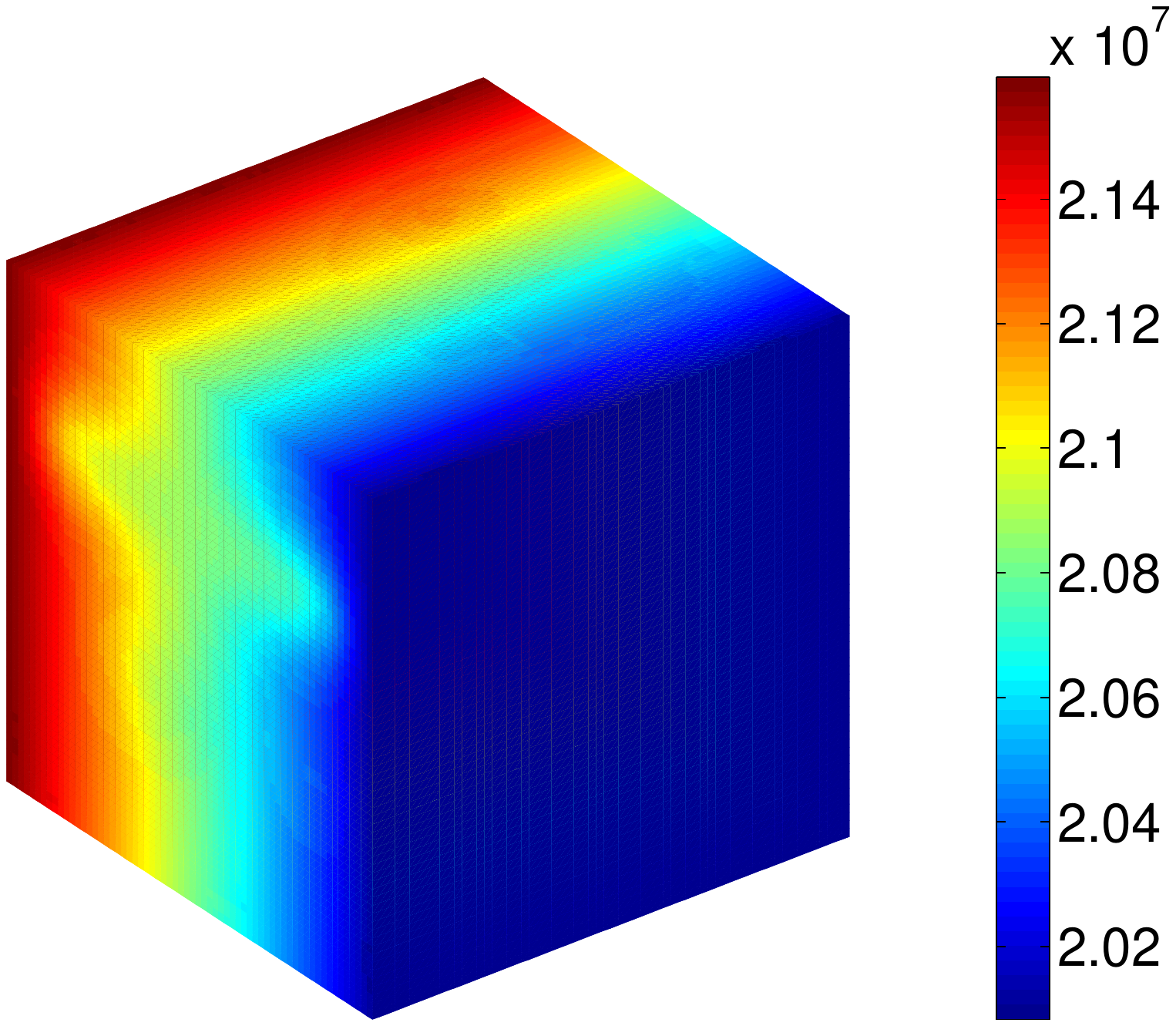}}	
	\subfigure[Reference solution at Day 105]{\includegraphics[trim={1cm 6cm 1cm 6cm},clip,width=3in]{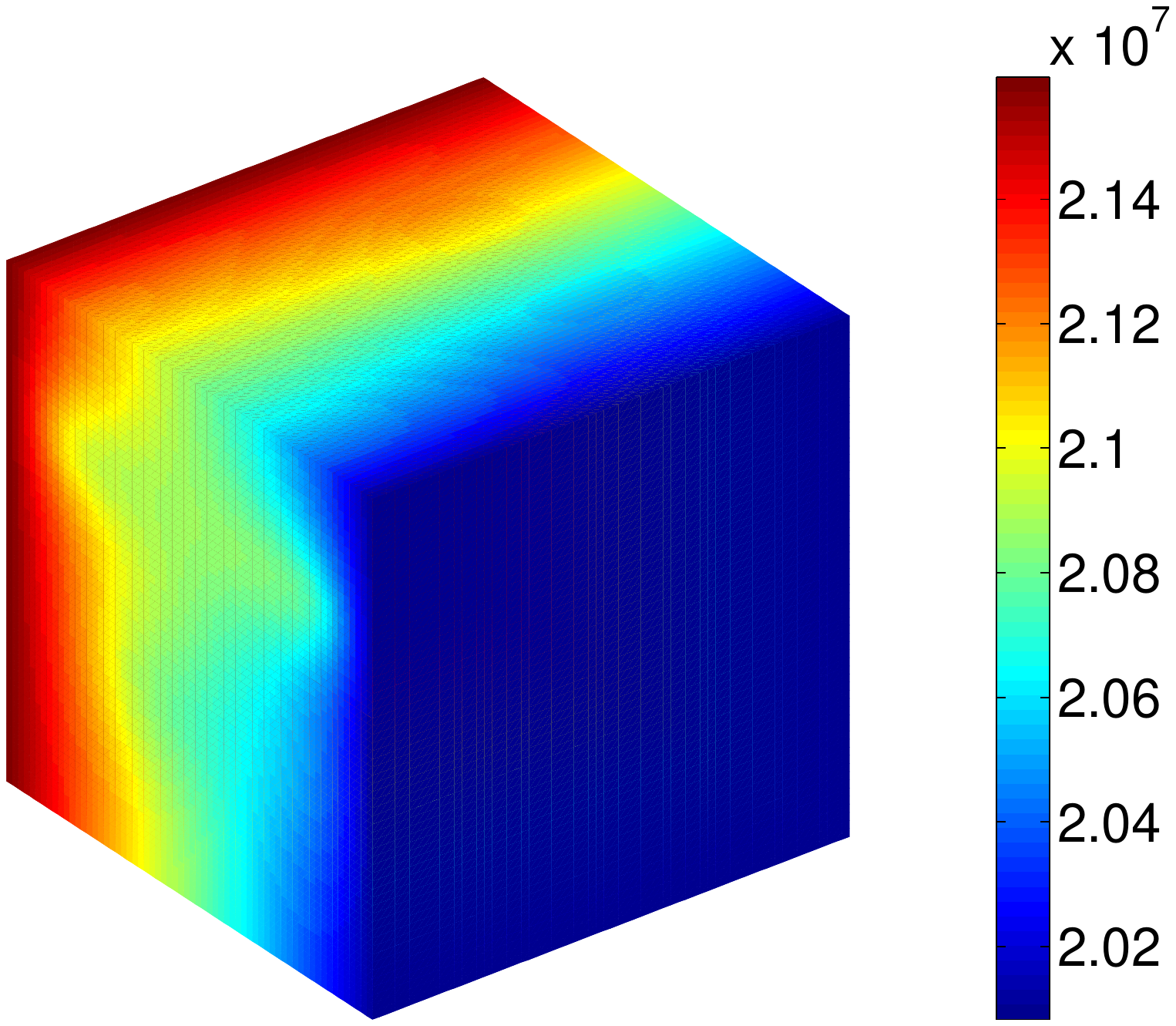}}	
	\subfigure[GMsFEM solution with 5+1 bases at Day 105]{\includegraphics[trim={1cm 6cm 1cm 6cm},clip,width=3in]{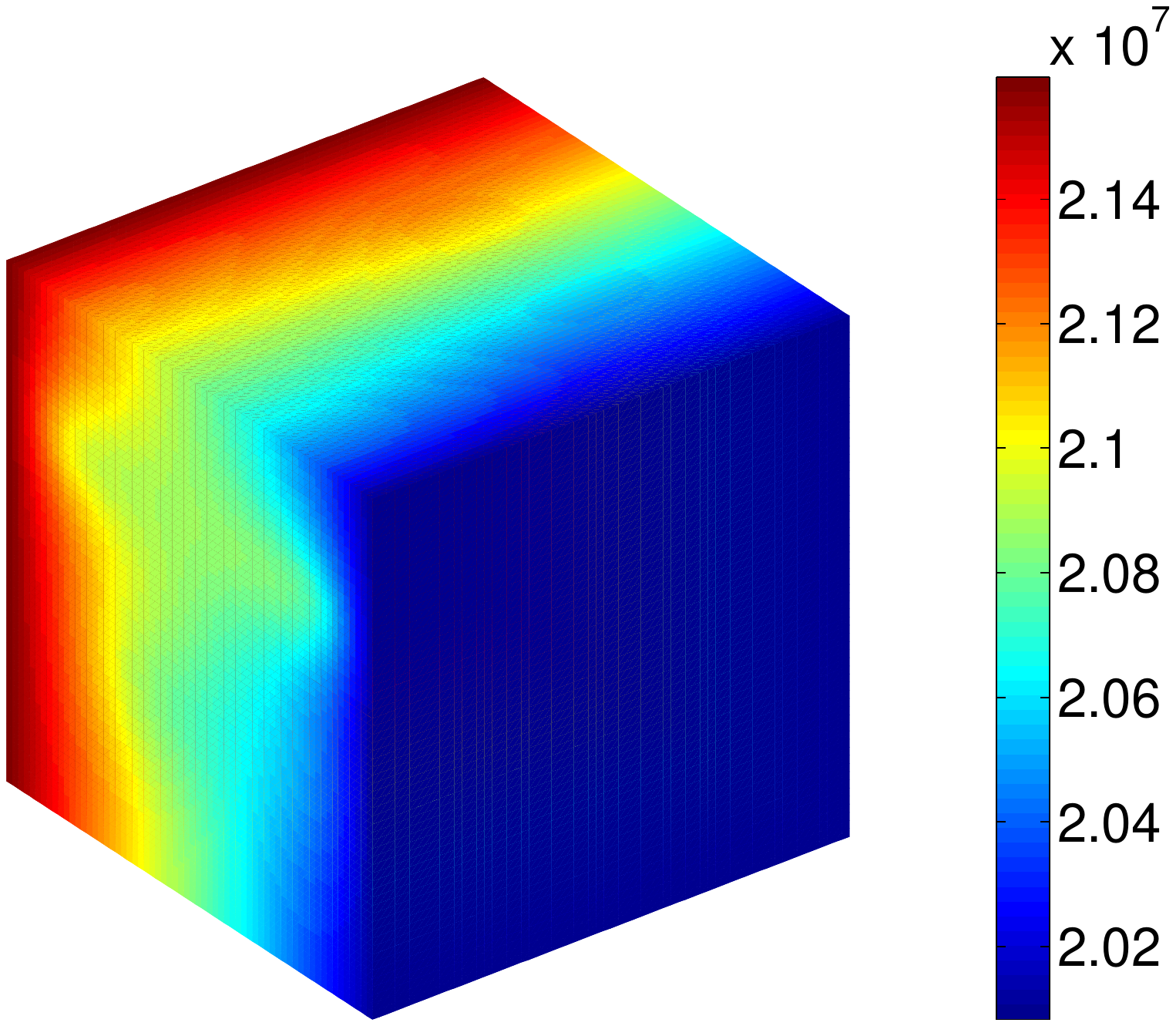}}
	
	\caption{Fine-scale reference solution and GMsFEM solution with 5+1 bases at different time instants, $K_2$, mixed boundary condition.}
	\label{fig:k2_mixedbd}
\end{figure}

\begin{figure}[!htb]
	\centering
	\subfigure[Reference solution at Day 35]{\includegraphics[trim={1cm 6cm 1cm 6cm},clip,width=3in]{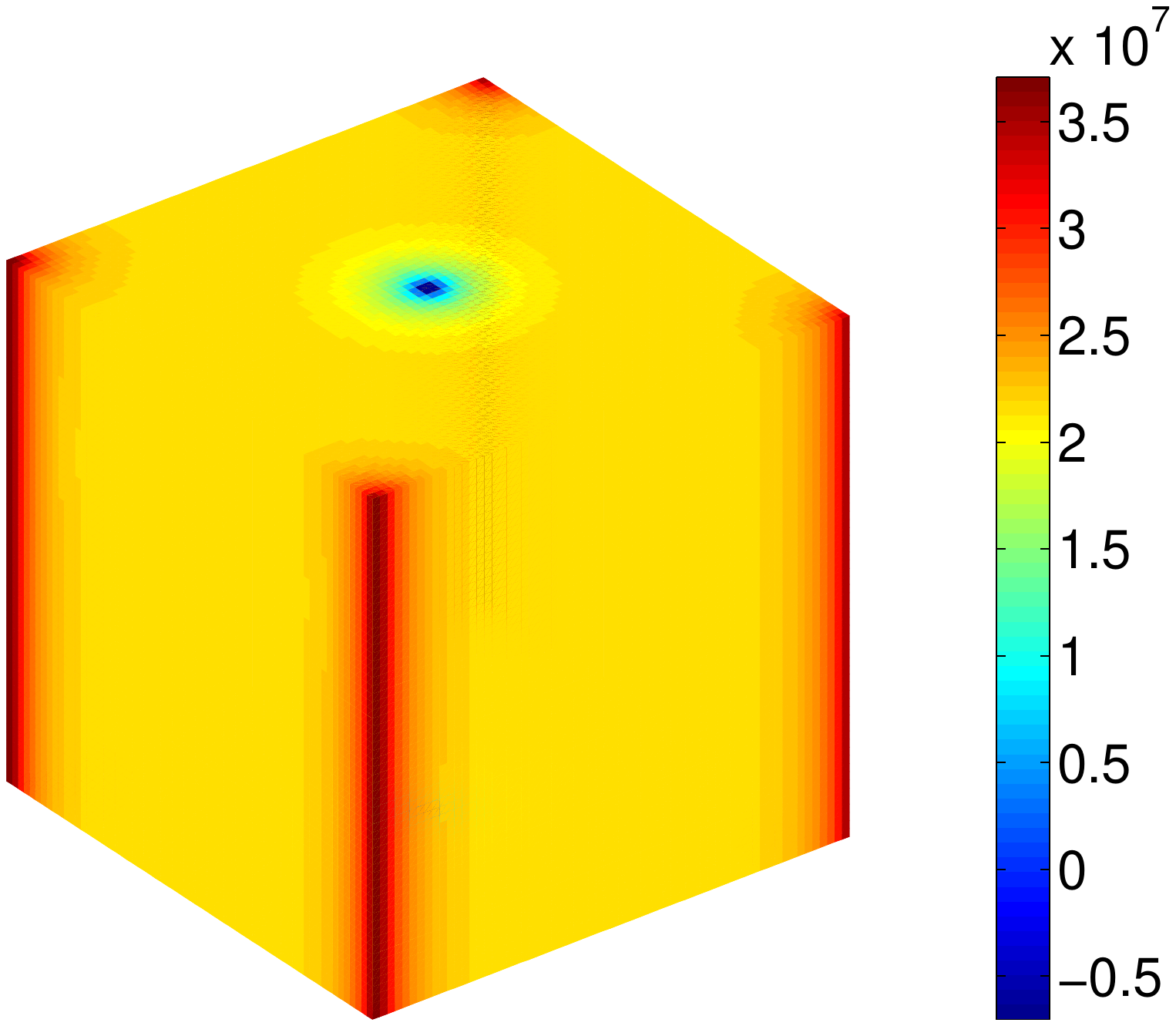}}
		\subfigure[GMsFEM solution with 5+1 bases at Day 35]{\includegraphics[trim={1cm 6cm 1cm 6cm},clip,width=3in]{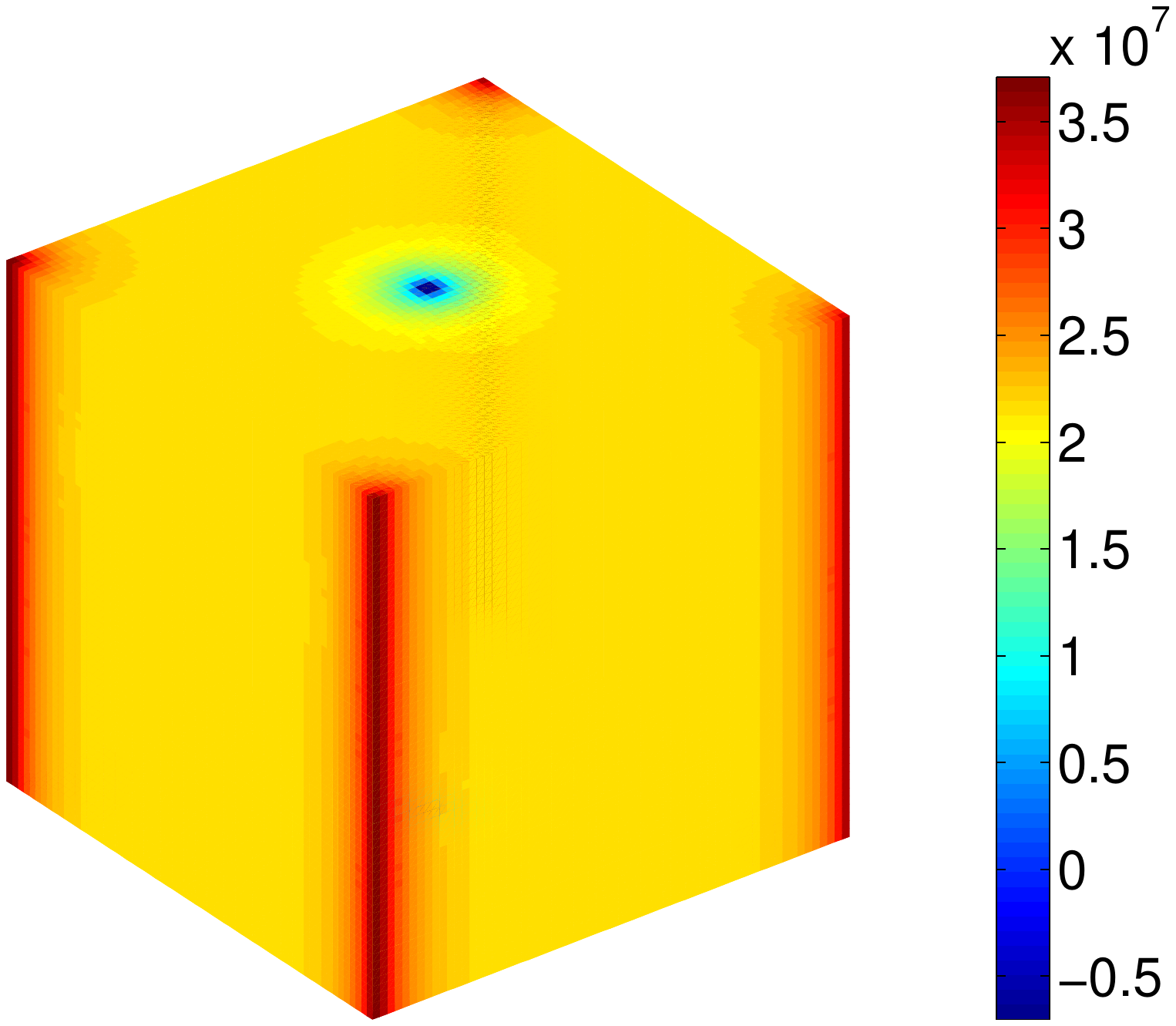}}
	\subfigure[Reference solution at Day 105]{\includegraphics[trim={1cm 6cm 1cm 6cm},clip,width=3in]{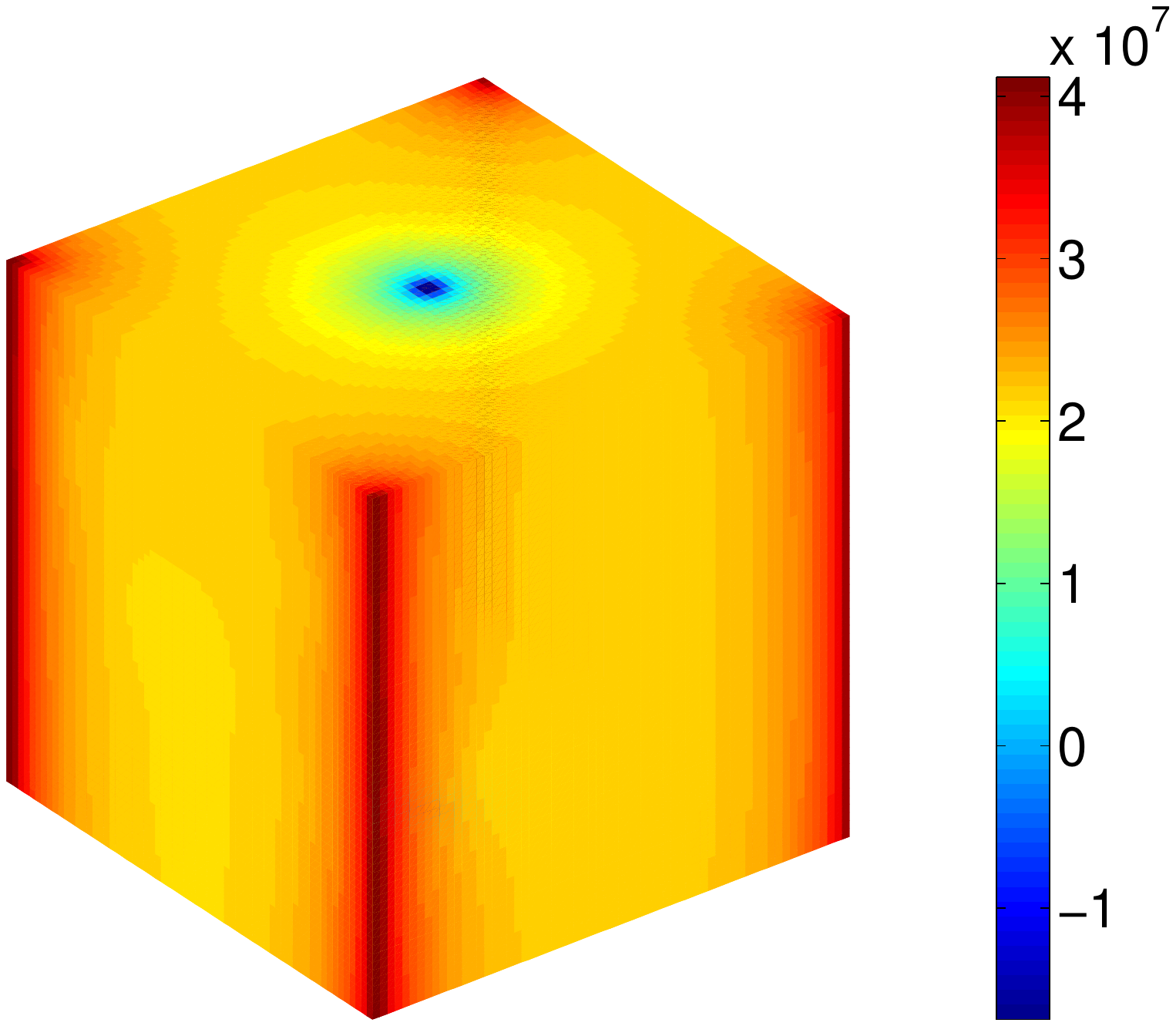}}	
	\subfigure[GMsFEM solution with 5+1 bases at Day 105]{\includegraphics[trim={1cm 6cm 1cm 6cm},clip,width=3in]{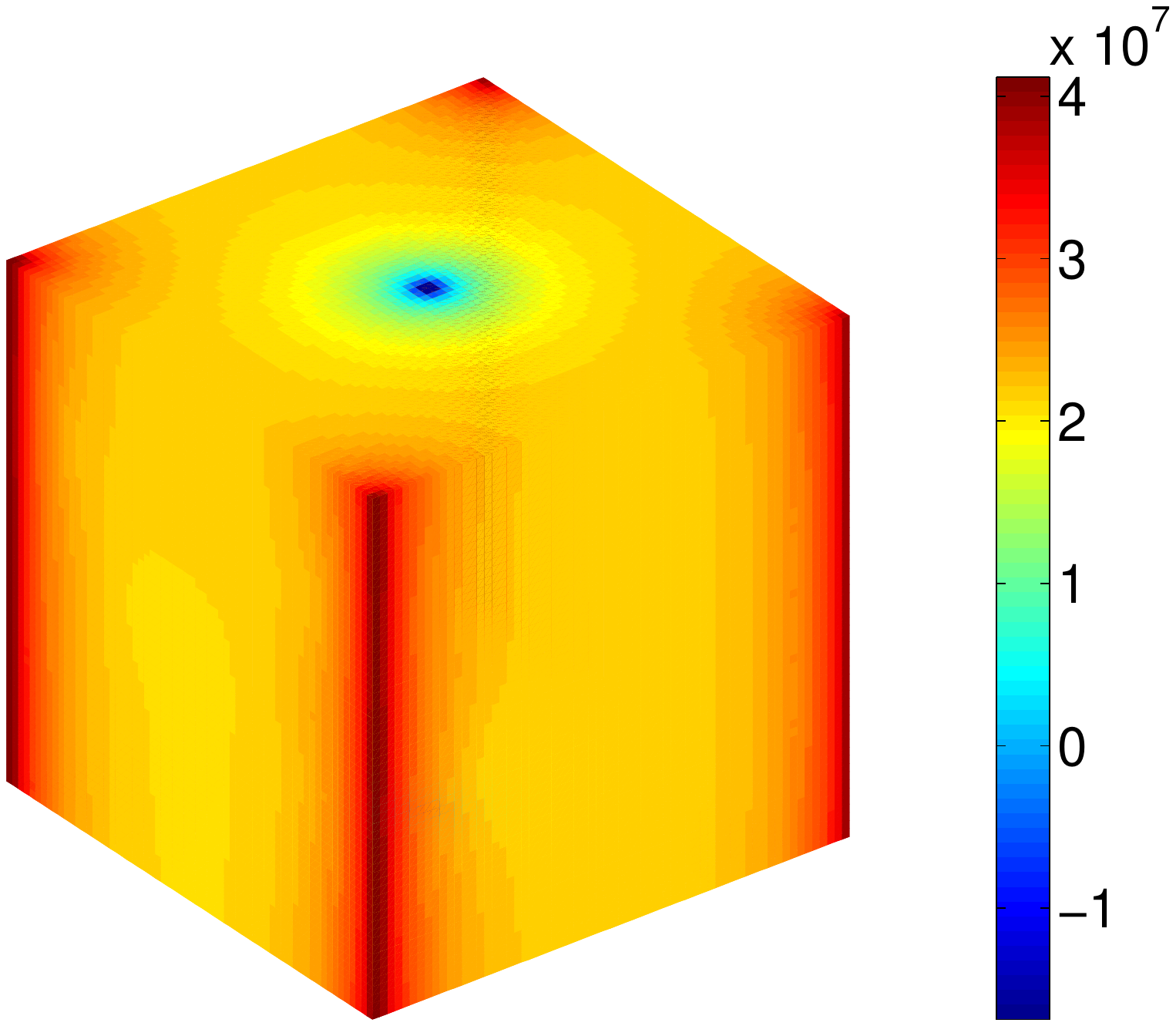}}		
	\caption{Fine-scale reference solution and GMsFEM solution with 5+1 bases (3 updates) at different time instants, $K_2$, full zero Neumann boundary condition.}
	\label{fig:k2_neu}
\end{figure}

\begin{figure}[!htb]
	\centering
	\subfigure[Reference solution at Day 5]{\includegraphics[trim={1cm 6cm 1cm 6cm},clip,width=3in]{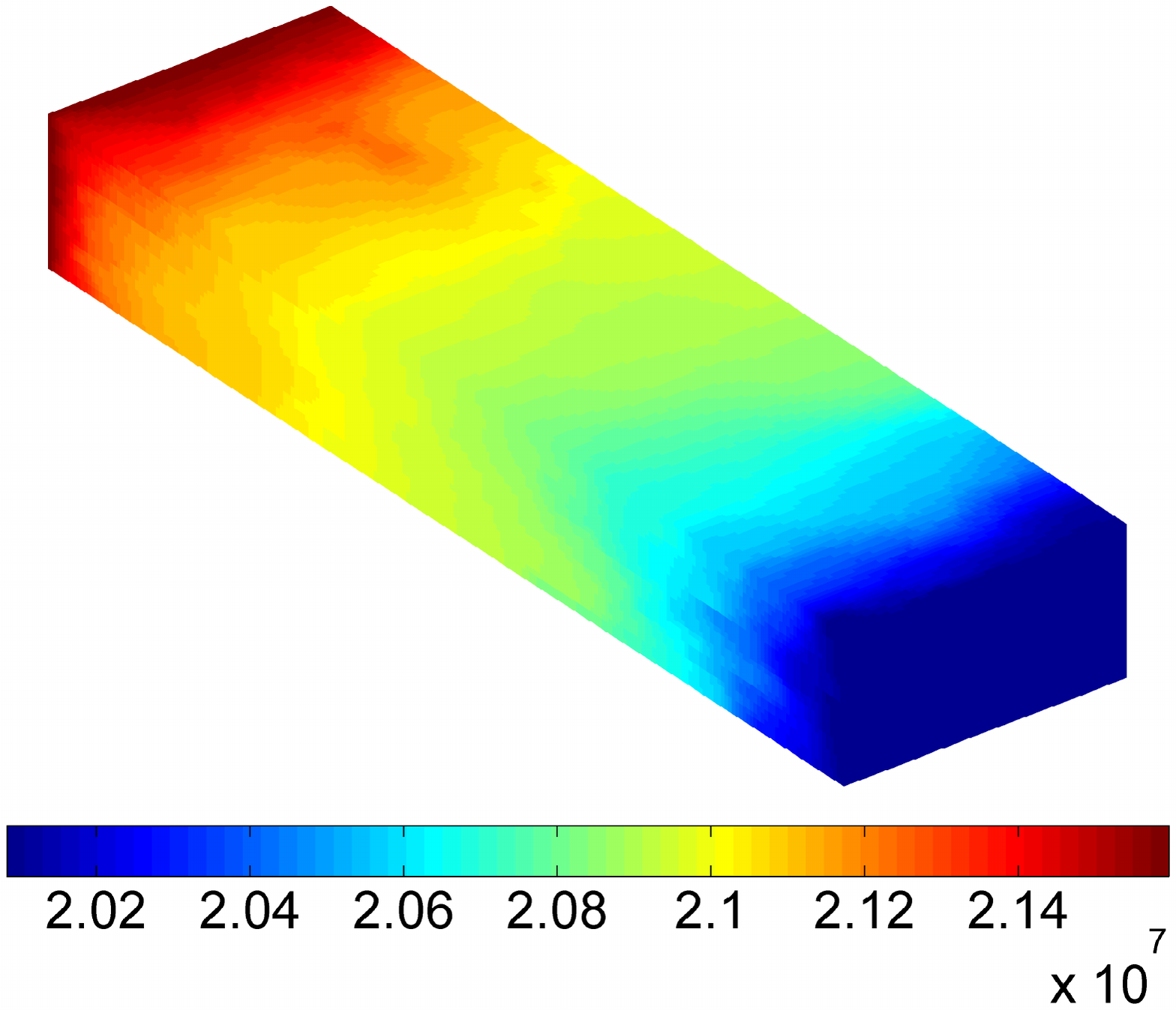}}
	\subfigure[GMsFEM solution with 5+1 bases at Day 5]{\includegraphics[trim={1cm 6cm 1cm 6cm},clip,width=3in]{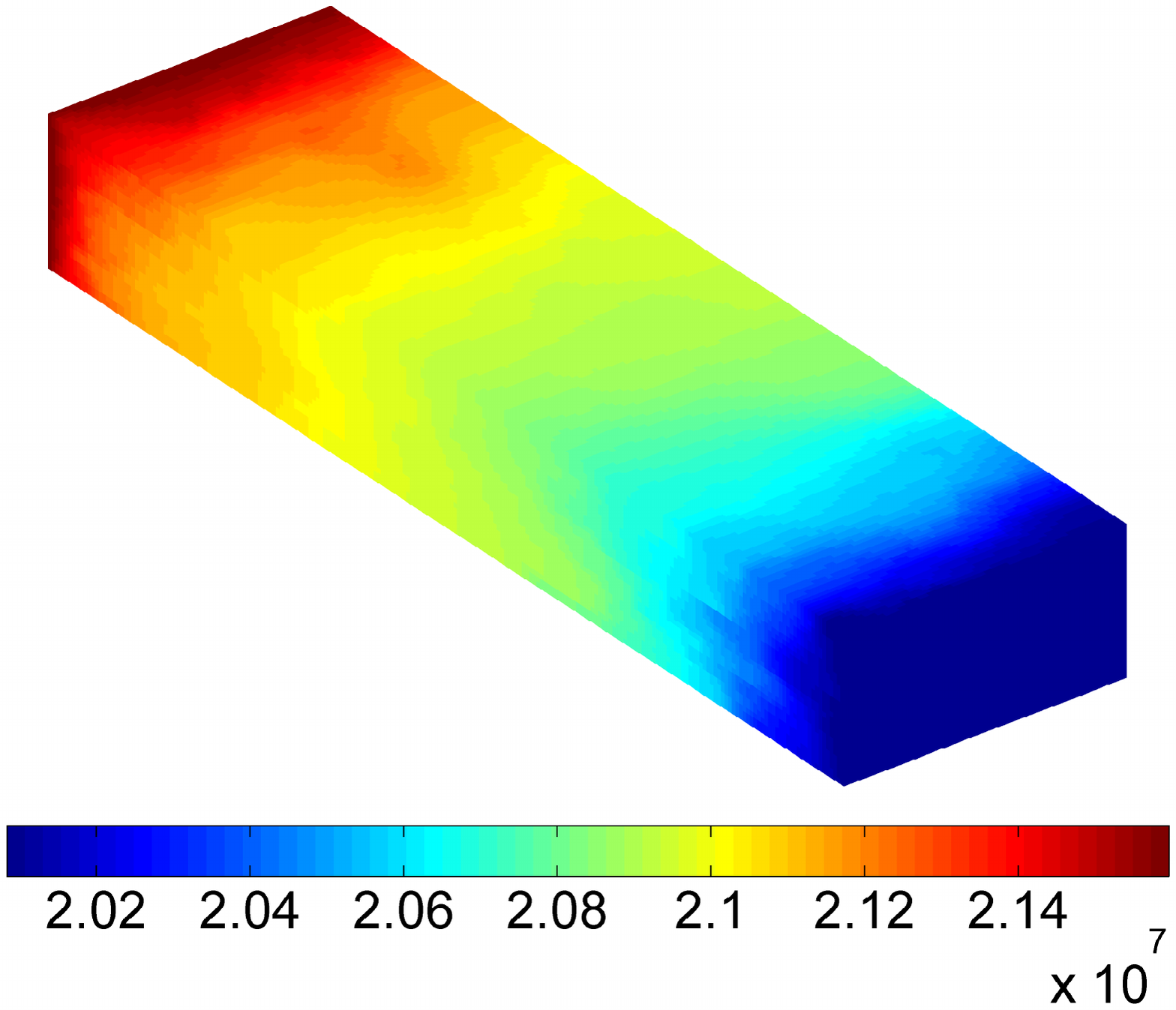}}	
	\subfigure[Reference solution at Day 15]{\includegraphics[trim={1cm 6cm 1cm 6cm},clip,width=3in]{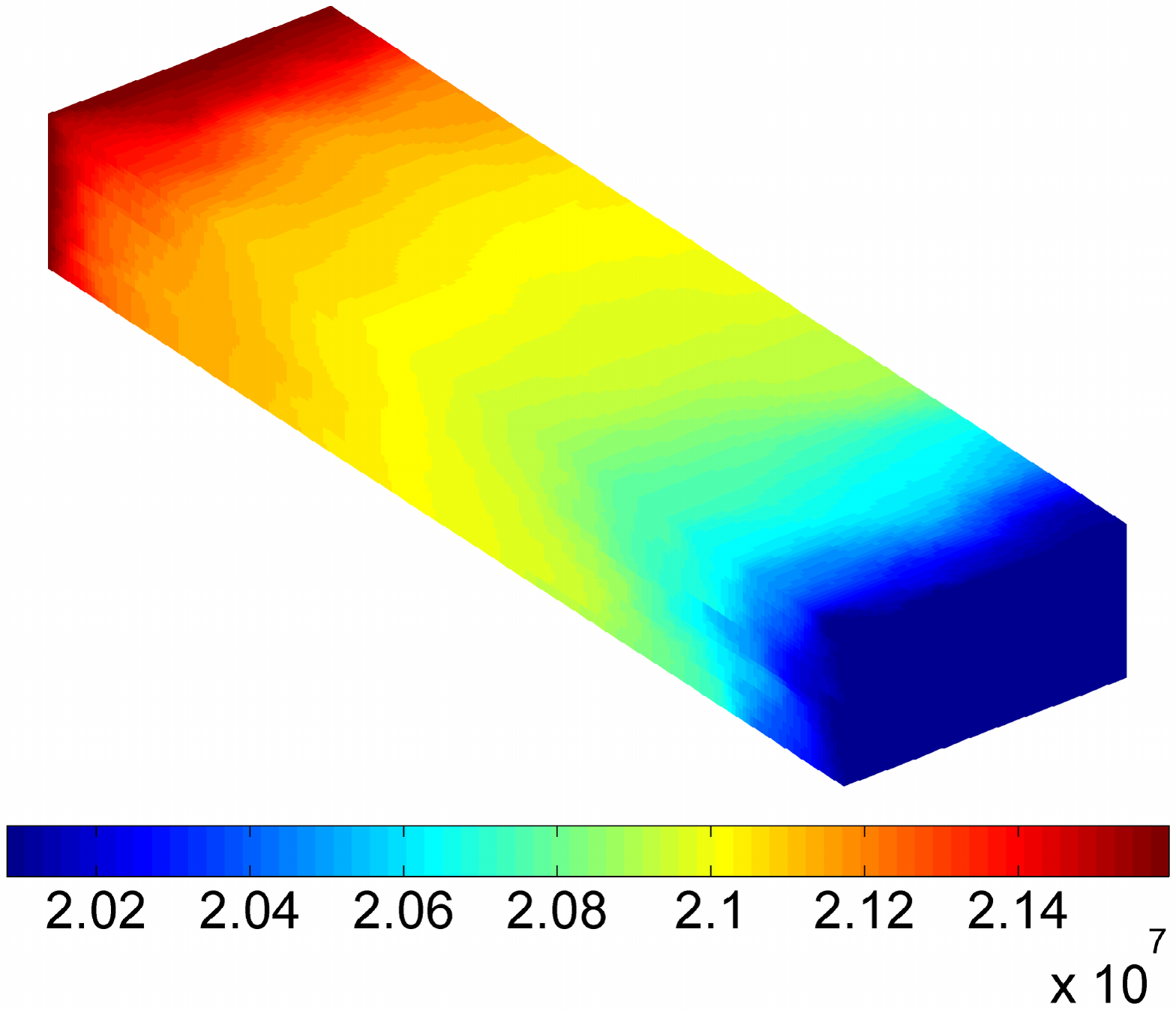}	}
	\subfigure[GMsFEM solution with 5+1 bases at Day 15]{\includegraphics[trim={1cm 6cm 1cm 6cm},clip,width=3in]{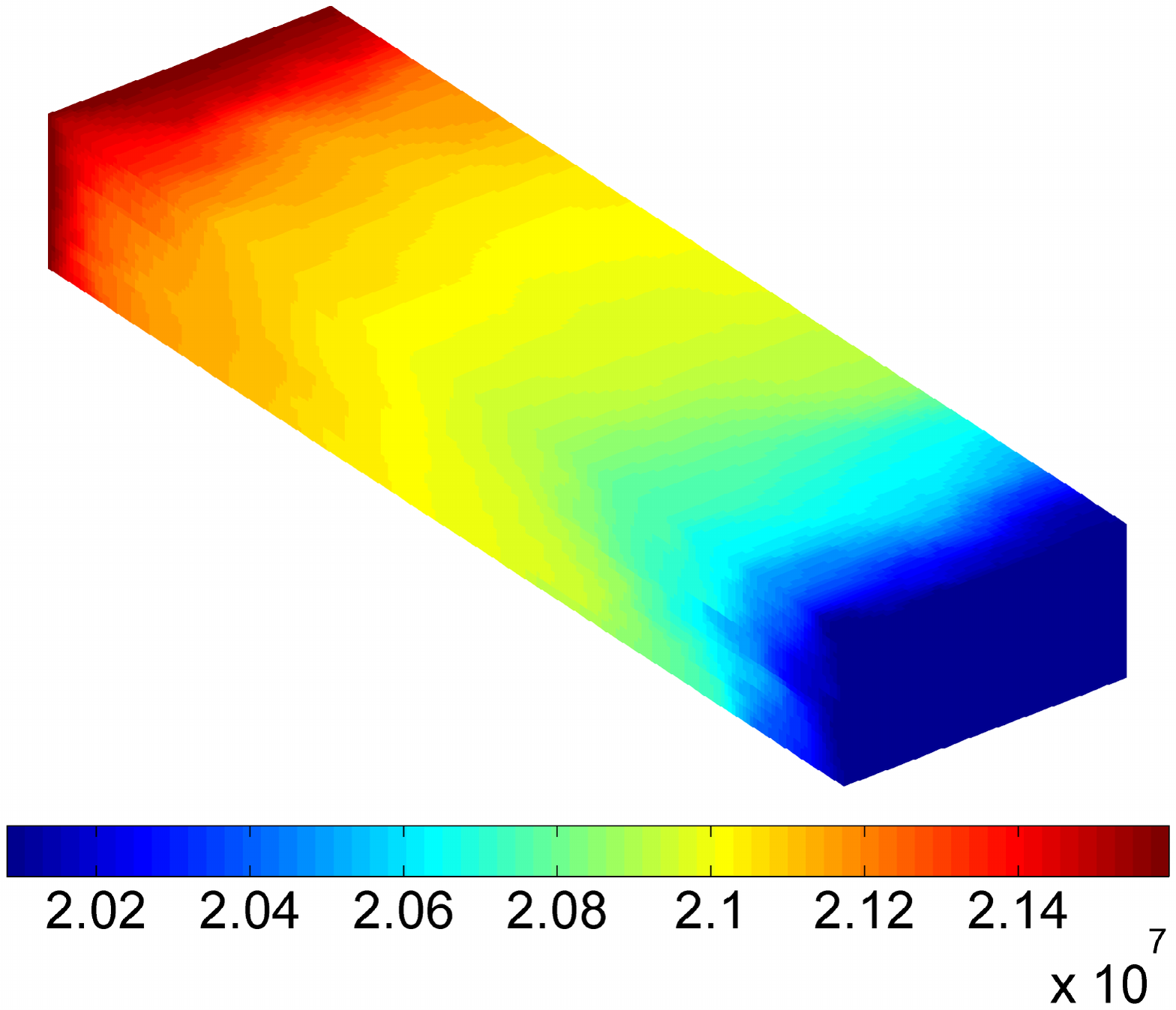}}		
	
	\caption{Fine-scale reference solution and GMsFEM solution with 5+1 bases (3 updates) at different time instants, $K_3$, mixed boundary condition.}
	\label{fig:spe3d_mixedbd}
\end{figure}

\begin{figure}[!htb]
	\centering
		\subfigure[Reference solution at Day 5]{\includegraphics[trim={1cm 6cm 1cm 6cm},clip,width=3in]{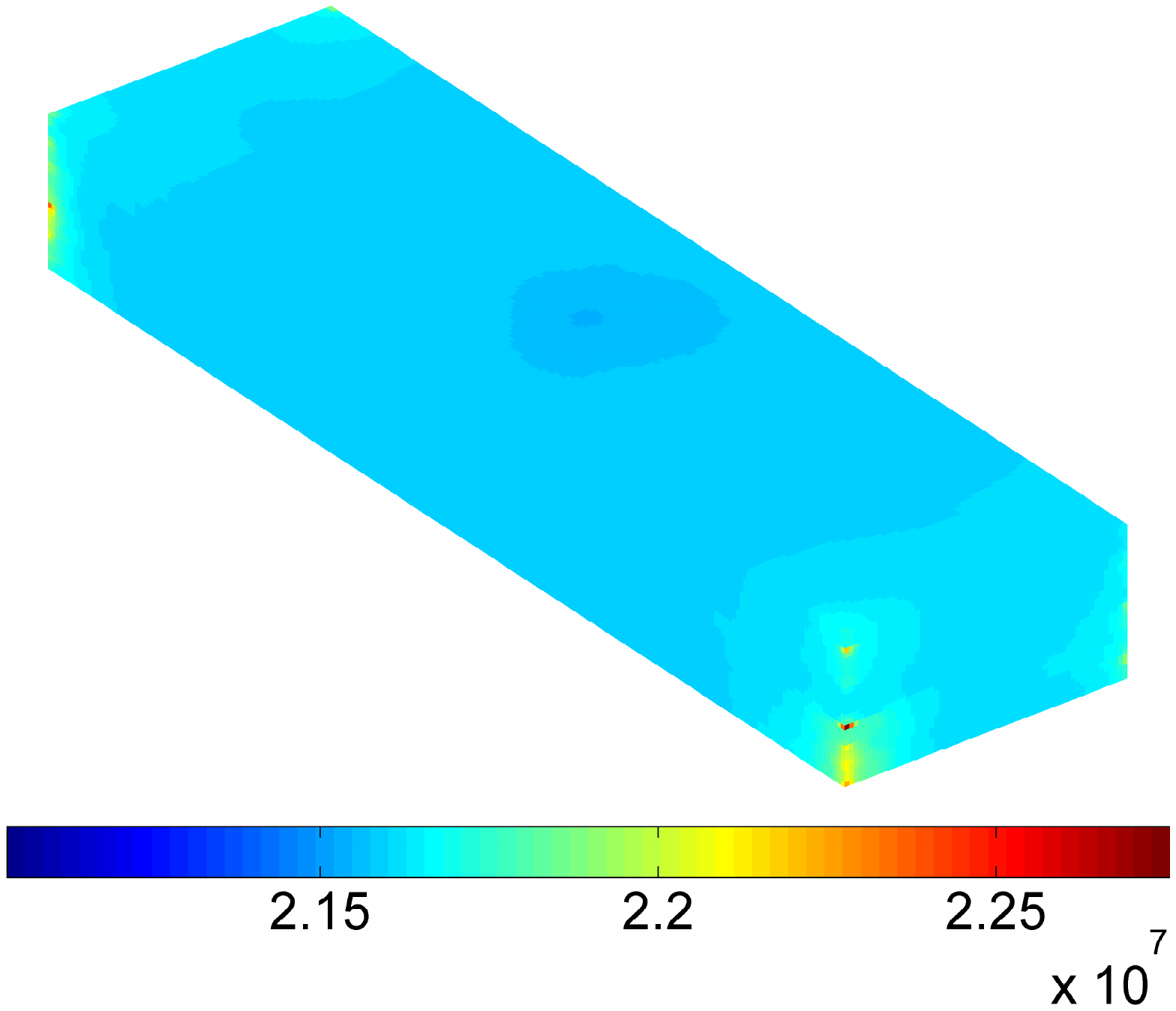}}
	\subfigure[GMsFEM solution with 5+1 bases at Day 5]{\includegraphics[trim={1cm 6cm 1cm 6cm},clip,width=3in]{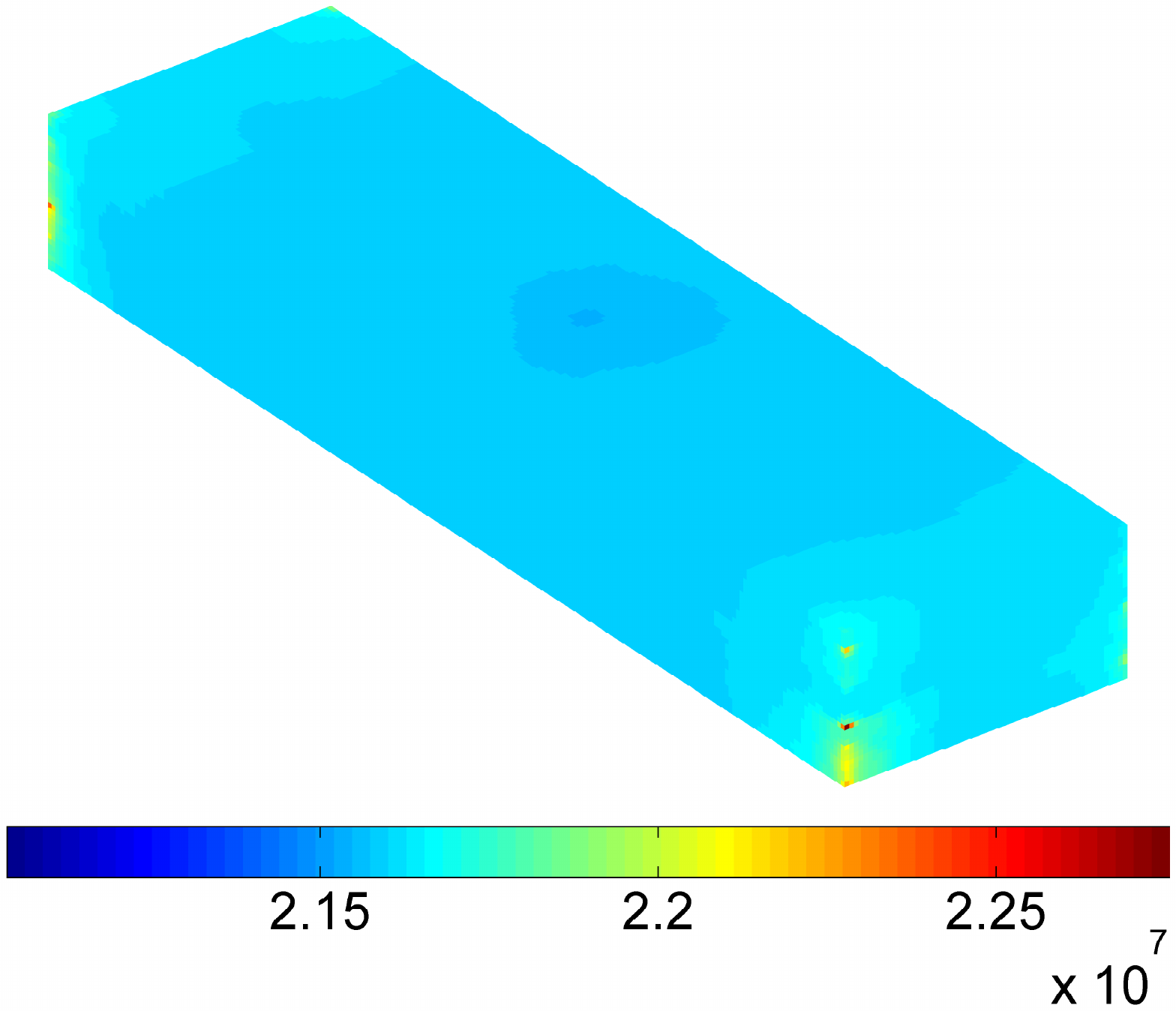}}		
		\subfigure[Reference solution at Day 15]{\includegraphics[trim={1cm 6cm 1cm 6cm},clip,width=3in]{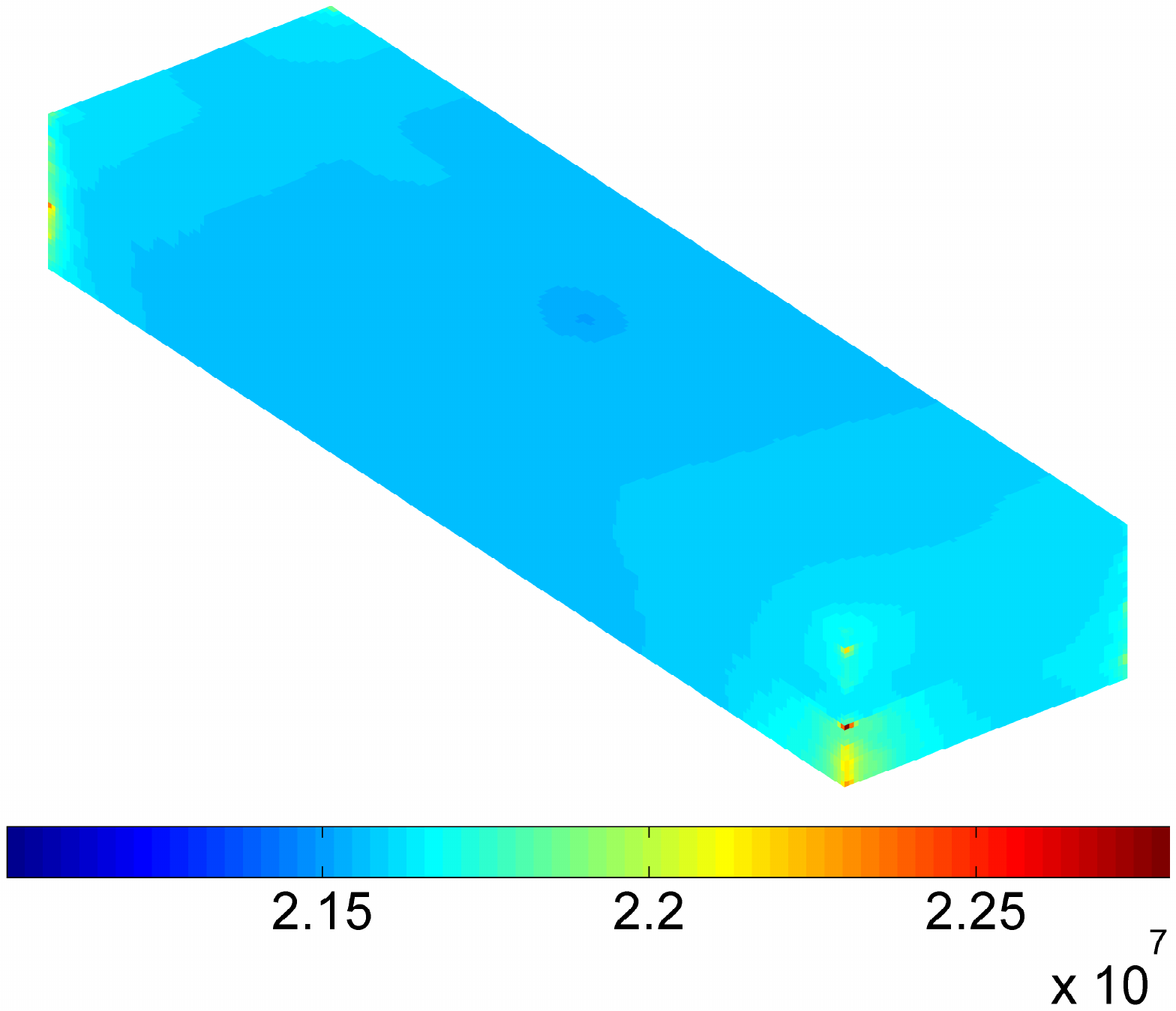}	}
	\subfigure[GMsFEM solution with 5+1 bases at Day 15]{\includegraphics[trim={1cm 6cm 1cm 6cm},clip,width=3in]{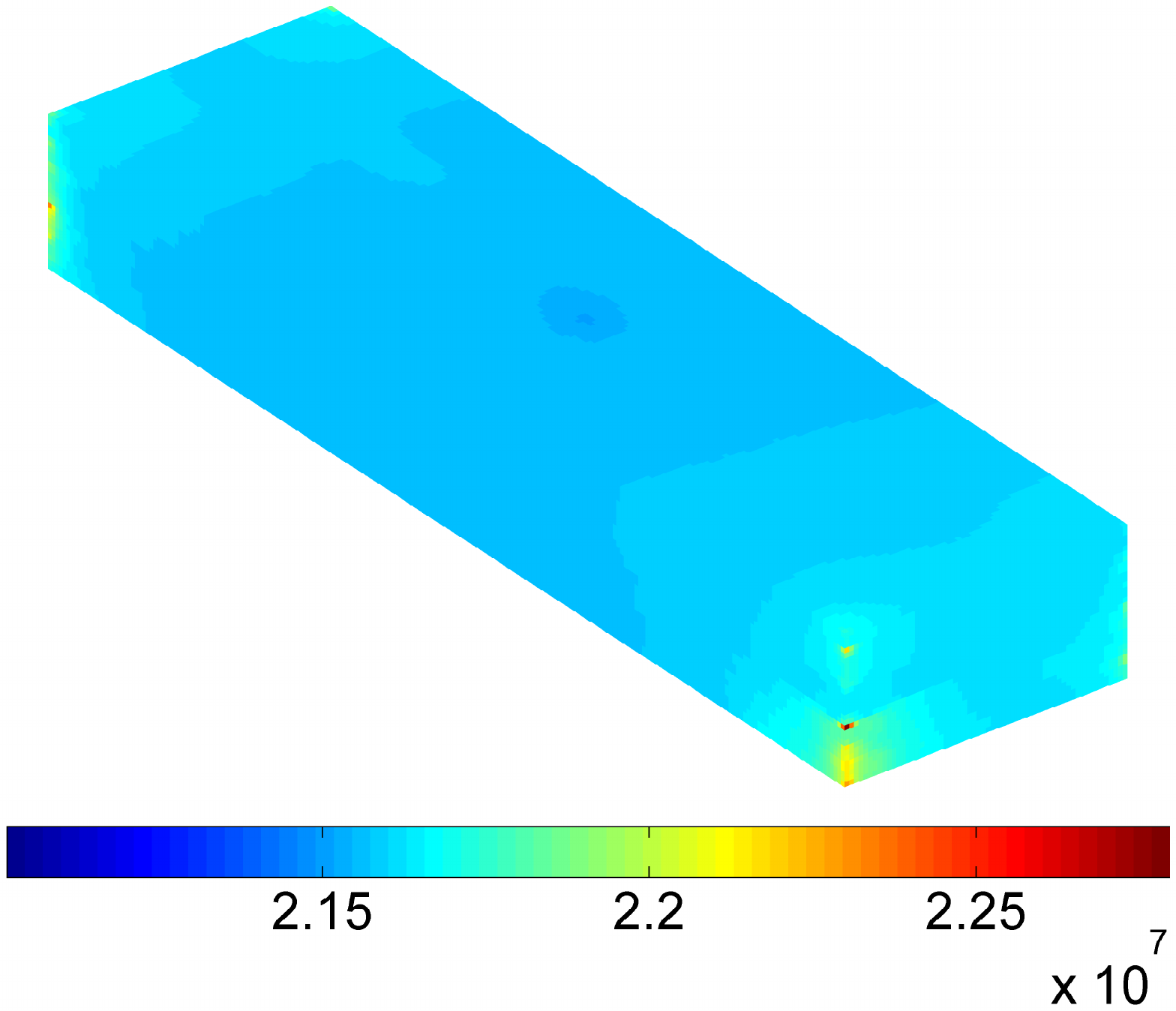}}	
	\caption{Fine-scale reference solution and GMsFEM solution with 5+1 bases (3 updates) at different time instants, $K_3$, full zero Neumann boundary condition.}
	\label{fig:spe3d_neubd}
\end{figure}

\begin{table}
	\centering \begin{tabular}{|c|c|c|c|c|c|c|c|}\hline
		Nb& Dim&$T_\text{basis}$&$T_\text{ass}$&$T_\text{solve}$& $e_{L^2}$& $e_{H^1}$  \tabularnewline\hline
		/&274,625&/& 123.8 &2406.7&/&/    \tabularnewline\hline
		4+0 & 2916 & 10.0 & 122.5 & 64.8 & 3.09e-04 & 1.37e-01     \tabularnewline\hline
		6+0 & 4374 & 12.2 & 123.5 & 121.3 & 2.05e-04 & 1.06e-01   \tabularnewline\hline
		8+0 & 5832 & 14.6 & 124.2 & 194.1 & 1.52e-04 & 8.61e-02\tabularnewline\hline
		2+1 & 2187 & 12.5 & 125.2 & 43.7 & 4.25e-04 & 1.65e-01   \tabularnewline\hline
		3+1 & 2916 & 14.1 & 124.8 & 68.5 & 3.67e-04 & 1.54e-01    \tabularnewline\hline
		3+2 & 3645 & 20.6 & 112.8 & 100.4 & 2.49e-04 & 1.10e-01\tabularnewline\hline
		4+1 & 3645 & 16.3 & 125.7 & 98.0 & 2.28e-04 & 1.07e-01    \tabularnewline\hline	
		4+2 & 4374 & 22.6 & 125.2 & 129.6 & 1.92e-04 & 9.23e-02    \tabularnewline\hline
		5+1 & 4374 & 18.4 & 125.4 & 130.6 & 1.96e-04 & 9.61e-02     \tabularnewline\hline
		4+1(1 update) & 3645 & 21.6 & 127.0 & 98.5 & 1.60e-04 & 7.19e-02   \tabularnewline\hline		
		4+1(3 updates) & 3645 & 32.4 & 125.3 & 98.0 & 1.11e-04 & 4.41e-02  \tabularnewline\hline		
		5+1(1 update) & 4374 & 24.6 & 126.1 & 129.9 & 1.36e-04 & 6.46e-02   \tabularnewline\hline
		5+1(3 updates)  & 4374 & 37.8 & 126.1 & 130.1 & 9.30e-05 & 3.95e-02    \tabularnewline\hline

	\end{tabular}
	\caption{Computational performance comparison between the reference solution and GMsFEM solution, $K_1$, mixed boundary condition.}
	\label{ta:k1m}
\end{table}

\begin{table}
	\centering \begin{tabular}{|c|c|c|c|c|c|c|c|}\hline
		Nb& Dim&$T_\text{basis}$&$T_\text{ass}$&$T_\text{solve}$& $e_{L^2}$& $e_{H^1}$  \tabularnewline\hline
	/&274,625&/&  259.3&4898.5&/&/    \tabularnewline\hline		
4+0 & 2916 & 10.5 & 203.1 & 100.5 & 7.62e-03 & 3.68e-01   \tabularnewline\hline	
6+0 & 4374 & 12.3 & 203.2 & 195.8 & 6.22e-03 & 3.28e-01    \tabularnewline\hline
8+0 & 5832 & 15.0 & 212.7 & 314.9 & 5.45e-03 & 3.04e-01    \tabularnewline\hline
 2+1 & 2187 & 11.7 & 240.0 & 81.7 & 3.99e-03 & 1.57e-01    \tabularnewline\hline
 2+2 & 2916 & 16.2 & 240.5 & 145.2 & 3.34e-03 & 1.20e-01\tabularnewline\hline
3+1 & 2916 & 13.5 & 239.8 & 128.1 & 2.78e-03 & 1.16e-01     \tabularnewline\hline
 3+2 & 3645 & 18.8 & 245.6 & 196.5 & 2.43e-03 & 9.38e-02     \tabularnewline\hline
4+2 & 4374 & 21.8 & 240.3 & 257.6 & 1.83e-03 & 7.54e-02    \tabularnewline\hline
4+1 &  3645 & 15.7 & 246.2 & 185.0 & 2.08e-03 & 9.52e-02   \tabularnewline\hline					
 5+1 & 4374 & 17.8 & 241.7 & 247.1 & 1.33e-03 & 7.03e-02   \tabularnewline\hline
 4+1(1 update)  & 3645 & 20.6 & 222.2 & 170.2 & 1.34e-03 & 6.71e-02\tabularnewline\hline
    4+1 (3 updates)& 3645 & 30.6 & 229.6 & 176.3 & 7.86e-04 & 4.32e-02		\tabularnewline\hline
  5+1(1 update) & 4374 & 23.8 & 221.7 & 225.9 & 8.90e-04 & 5.07e-02	\tabularnewline\hline
  5+1 (3 updates)& 4374 & 36.5 & 233.0 & 239.2 & 5.50e-04 & 3.36e-02	\tabularnewline\hline
	\end{tabular}
	\caption{Computational performance comparison between the reference solution and GMsFEM solution, $K_1$, full zero Neumann boundary condition.}
	\label{ta:k1n}
\end{table}

\begin{table}
	\centering \begin{tabular}{|c|c|c|c|c|c|c|c|}\hline
		Nb& Dim&$T_\text{basis}$&$T_\text{ass}$&$T_\text{solve}$& $e_{L^2}$& $e_{H^1}$  \tabularnewline\hline
		/&274,625&/& 112.0 &2429.0&/&/    \tabularnewline\hline	
		4+0 & 2916 & 10.2 & 121.5 & 65.2 & 1.97e-04 & 1.11e-01   \tabularnewline\hline		
		6+0 & 4374 & 12.2 & 123.3 & 120.7 & 8.65e-05 & 6.12e-02	\tabularnewline\hline
		8+0 & 5832 & 14.2 & 121.0 & 190.8 & 5.98e-05 & 4.83e-02\tabularnewline\hline
		2+1 & 2187 & 12.0 & 123.6 & 40.8 & 8.60e-04 & 2.90e-01     \tabularnewline\hline
		3+1 & 2916 & 13.6 & 123.1 & 65.4 & 1.86e-04 & 1.06e-01  \tabularnewline\hline
		3+2 & 3645 & 19.3 & 123.2 & 100.0 & 1.31e-04 & 7.30e-02    \tabularnewline\hline
		4+2 & 4374 & 22.4 & 121.9 & 127.7 & 8.27e-05 & 5.39e-02   \tabularnewline\hline	
		5+1 & 4374 & 17.7 & 124.0 & 128.4 & 7.35e-05 & 5.29e-02    \tabularnewline\hline
		4+1 & 3645 & 15.7 & 123.5 & 95.9 & 1.04e-04 & 6.88e-02   \tabularnewline\hline
		4+1(1 update) & 3645 & 21.4 & 123.2 & 96.6 & 7.82e-05 & 4.93e-02    \tabularnewline\hline
		4+1(3 updates) & 3645 & 32.2 & 123.9 & 97.2 & 5.67e-05 & 3.21e-02   \tabularnewline\hline
		5+1(1 update) & 4374 & 24.4 & 123.5 & 127.9 & 5.51e-05 & 3.76e-02    \tabularnewline\hline
		5+1(3 updates) & 4374 & 37.2 & 125.2 & 128.5 & 4.12e-05 & 2.44e-02   \tabularnewline\hline		
	\end{tabular}
	\caption{Computational performance comparison between the reference solution and GMsFEM solution, $K_2$, mixed boundary condition.}
	\label{ta:k2m}
\end{table}

\begin{table}
	\centering \begin{tabular}{|c|c|c|c|c|c|c|c|}\hline
		Nb& Dim&$T_\text{basis}$&$T_\text{ass}$&$T_\text{solve}$& $e_{L^2}$& $e_{H^1}$  \tabularnewline\hline
		/&274,625&/& 287.5 &5610.9&/&/    \tabularnewline\hline
		4+0 & 2916 & 9.8 & 222.5 & 110.2 & 1.81e-02 & 5.46e-01     \tabularnewline\hline
		6+0 & 4374 & 12.0 & 241.2 & 233.6 & 1.39e-02 & 4.81e-01    \tabularnewline\hline
		8+0 & 5832 & 13.9 & 241.6 & 372.7 & 1.18e-02 & 4.39e-01    \tabularnewline\hline
		2+1 & 2187 & 11.5 & 278.7 & 88.1 & 1.97e-02 & 2.50e-01   \tabularnewline\hline	
		3+1 & 2916 & 13.4 & 277.6 & 148.6 & 6.83e-03 & 1.80e-01    \tabularnewline\hline
		3+2 & 3645 & 18.6 & 280.8 & 228.5 & 4.42e-03 & 9.28e-02\tabularnewline\hline
		4+1 & 3645 & 15.7 & 279.4 & 214.6 & 5.16e-03 & 1.59e-01    \tabularnewline\hline
		4+2 & 4374 & 21.6 & 278.0 & 297.7 & 2.80e-03 & 7.26e-02 \tabularnewline\hline
		5+1  & 4374 & 18.0 & 288.0 & 292.0 & 4.04e-03 & 1.40e-01  \tabularnewline\hline
		4+1(1 update)& 3645 & 21.1 & 258.0 & 196.8 & 3.71e-03 & 1.17e-01    \tabularnewline\hline
		4+1(3 updates) & 3645 & 30.5 & 255.2 & 195.5 & 2.26e-03 & 7.59e-02     \tabularnewline\hline		
		5+1(1 update)& 4374 & 23.3 & 260.1 & 265.7 & 2.91e-03 & 1.04e-01    \tabularnewline\hline
		5+1(3 updates) & 4374 & 35.7 & 262.3 & 266.0 & 1.76e-03 & 6.70e-02    \tabularnewline\hline		
	\end{tabular}
	\caption{Computational performance comparison between the reference solution and GMsFEM solution, $K_2$, full zero Neumann boundary condition.}
	\label{ta:k2n}
\end{table}

\begin{table}
	\centering \begin{tabular}{|c|c|c|c|c|c|c|c|}\hline
		Nb& Dim&$T_\text{basis}$&$T_\text{ass}$&$T_\text{solve}$& $e_{L^2}$& $e_{H^1}$  \tabularnewline\hline
		/&417,911&/& 246.3& 3264.5&/&/    \tabularnewline\hline
		4+0 & 2576 & 19.2 & 185.1 & 92.4 & 2.89e-04 & 2.32e-01	 \tabularnewline\hline	
		6+0 & 3864 & 23.5 & 194.0 & 172.7 & 2.03e-04 & 1.54e-01  \tabularnewline\hline
		8+0 & 5152 & 26.6 & 194.7 & 271.8 & 1.53e-04 & 1.16e-01   \tabularnewline\hline
		2+1 & 1932 & 23.3 & 197.1 & 64.2 & 2.53e-04 & 1.16e-01    \tabularnewline\hline	
		3+1 & 2576 & 26.6 & 194.4 & 97.8 & 1.86e-04 & 9.33e-02     \tabularnewline\hline
		3+2 & 3220 & 35.3 & 187.5 & 140.0 & 1.45e-04 & 5.55e-02   \tabularnewline\hline
		4+1 & 3220 & 28.9 & 197.7 & 139.3 & 1.48e-04 & 7.67e-02    \tabularnewline\hline
		4+2 & 3864 & 40.3 & 199.4 & 189.6 & 1.25e-04 & 4.62e-02    \tabularnewline\hline	
		5+1 & 3864 & 33.0 & 195.1 & 183.8 & 1.14e-04 & 6.42e-02   \tabularnewline\hline			
		4+1(1 update) & 3220 & 56.6 & 194.1 & 138.5 & 9.75e-05 & 3.83e-02    \tabularnewline\hline
		4+1(3 updates)& 3220 & 38.3 & 198.8 & 138.5 & 1.10e-04 & 5.62e-02    \tabularnewline\hline	
		5+1(1 update)& 3864 & 43.5 & 195.6 & 181.1 & 8.72e-05 & 4.72e-02   \tabularnewline\hline
		5+1(3 updates) & 3864 & 63.9 & 203.2 & 190.0 & 7.68e-05 & 3.30e-02    \tabularnewline\hline
	\end{tabular}
	\caption{Computational performance comparison between the reference solution and GMsFEM solution, $K_3$, mixed boundary condition.}
	\label{ta:k3m}
\end{table}

\begin{table}
	\centering \begin{tabular}{|c|c|c|c|c|c|c|c|}\hline
		Nb& Dim&$T_\text{basis}$&$T_\text{ass}$&$T_\text{solve}$& $e_{L^2}$& $e_{H^1}$  \tabularnewline\hline
		/&417,911&/&258.7& 3515.4&/&/    \tabularnewline\hline
		4+0 & 2576 & 19.3 & 195.2 & 90.6 & 1.94e-04 & 5.95e-01   \tabularnewline\hline
		6+0 & 3864 & 23.0 & 199.9 & 175.2 & 1.72e-04 & 5.30e-01    \tabularnewline\hline
		8+0 & 5152 & 27.9 & 194.8 & 278.9 & 1.65e-04 & 5.03e-01   \tabularnewline\hline
		3+1 & 2576 & 25.0 & 242.1 & 115.0 & 4.67e-05 & 1.22e-01   \tabularnewline\hline
		3+2 & 3220 & 33.4 & 254.3 & 177.5 & 4.00e-05 & 7.72e-02   \tabularnewline\hline	
		4+1 & 3220 & 28.3 & 243.0 & 169.0 & 3.18e-05 & 1.02e-01    \tabularnewline\hline
		4+2 & 3864 & 38.0 & 246.2 & 229.4 & 2.80e-05 & 6.27e-02    \tabularnewline\hline
		5+1 & 3864 & 32.0 & 248.7 & 225.2 & 2.34e-05 & 8.27e-02   \tabularnewline\hline
		4+1(1 update) & 3220 & 38.6 & 209.5 & 148.8 & 2.13e-05 & 6.97e-02   \tabularnewline\hline
		4+1(3 updates)& 3220 & 54.3 & 210.2 & 146.5 & 1.38e-05 & 4.30e-02   \tabularnewline\hline
		5+1(1 update) & 3864 & 42.0 & 218.8 & 197.9 & 1.53e-05 & 5.59e-02    \tabularnewline\hline
		5+1(3 updates) & 3864 & 64.5 & 214.5 & 194.1 & 9.58e-06 & 3.43e-02   \tabularnewline\hline				
	\end{tabular}
	\caption{Computational performance comparison between the reference solution and GMsFEM solution, $K_3$, full zero Neumann boundary condition.}
	\label{ta:k3n}
\end{table}

\section{Conclusion} \label{sec:conclusion}
We study the generalized multiscale finite element method for the highly heterogeneous
nonlinear single phase compressible flow.  We include the major ingredients of the
offline GMsFEM and adopt the residual driven GMsFEM. A comprehensive analysis is provided, which can guide future study. More specifically, the convergence error analysis for the semi-discrete scheme based on two types of snapshot spaces is performed, and the a posteriori error estimator is derived under the underlining discretization. Three representative 3D examples are offered to verify the efficiency and accuracy of the GMsFEM. Highly accurate solution with huge computational savings can be obtained. The results indicate that our proposed algorithm for the nonlinear single phase compressible flow is highly competitive among all the developed methods and could be a good candidate for practical applications.

\section*{Acknowledgment}

The research of Eric Chung is partially supported by the Hong Kong RGC General Research Fund (Project numbers 14304719 and 14302018).

\clearpage

	\bibliographystyle{plain}
\bibliography{reference}
\end{document}